\documentclass[a4paper, 12pt, oneside, notitlepage]{amsart}
\usepackage{amsmath,amssymb,amsthm,graphicx,mathrsfs,bbm,url}
\usepackage[margin=3cm]{geometry}
\usepackage{amsthm}
\usepackage{stmaryrd}
\usepackage{wrapfig}
\usepackage{enumitem}
\usepackage{mathtools}
\usepackage[all]{xy}
\usepackage[utf8]{inputenc}
\usepackage[usenames,dvipsnames]{color}
\usepackage[colorlinks=true,linkcolor=Red,citecolor=Green]{hyperref}
\usepackage[super]{nth}
\usepackage[open, openlevel=2, depth=3, atend]{bookmark}
\hypersetup{pdfstartview=XYZ}
\usepackage[font=footnotesize]{caption}
\usepackage{subcaption}
\usepackage{caption}
\usepackage{tikz}
\usepackage{setspace}
\usepackage[normalem]{ ulem }
\usepackage{soul}
\usetikzlibrary{decorations.pathreplacing}
\usetikzlibrary{arrows}
\usetikzlibrary{shapes.misc}
\usetikzlibrary{shapes.symbols}
\usetikzlibrary{patterns}
\usepackage{epstopdf}
 
\usepackage{hyperref}

\usepackage{listings}
\usepackage{color}

\definecolor{dkgreen}{rgb}{0,0.6,0}
\definecolor{gray}{rgb}{0.5,0.5,0.5}
\definecolor{mauve}{rgb}{0.58,0,0.82}

\theoremstyle{plain}
\newtheorem{theorem}{Theorem}[section]
\newtheorem*{theorem*}{Theorem}

\newtheorem{lemma}[theorem]{Lemma}
\newtheorem{proposition}[theorem]{Proposition}
\newtheorem{corollary}[theorem]{Corollary}

\newtheorem{conjecture}[theorem]{Conjecture}

\theoremstyle{definition}
\newtheorem{definition}[theorem]{Definition}
\newtheorem{example}[theorem]{Example}

\theoremstyle{remark}
\newtheorem{remark}[theorem]{Remark}

\numberwithin{equation}{section}

\newcommand{\C}{\mathbb{C}}
\newcommand{\R}{\mathbb{R}}

\newcommand{\Z}{\mathbb{Z}}
\newcommand{\F}{\mathbb{F}}

\newcommand{\M}{\mathcal{M}}

\newcommand{\V}{\mathbb{V}}
\newcommand{\WW}{\mathbf{W}}
\newcommand{\X}{\mathbf{X}}

\newcommand{\Ss}{\mathbb{S}}

\newcommand{\mc}{\mathcal}

\DeclareMathOperator{\Tr}{Tr}

\DeclareMathOperator{\rk}{rank}

\DeclareMathOperator{\id}{Id}
\DeclareMathOperator{\E}{\mathcal{E}}
\DeclareMathOperator{\A}{\mathbf{A}}
\DeclareMathOperator{\Hom}{Hom}

\DeclareMathOperator{\e}{\mathbf{e}}

\DeclareMathOperator{\End}{\mathrm{End}}

\newcommand{\be}{\begin{equation}}
\newcommand{\ee}{\end{equation}}

\def\beq{\begin{equation}}
\def\eeq{\end{equation}}

\def\beq{\begin{equation}}
\def\eeq{\end{equation}}
\def\bea{\begin{eqnarray*}}
\def\eea{\end{eqnarray*}}

\title[Connections, frame flows, and polynomial maps]{Isospectral connections, ergodicity of frame flows, and polynomial maps between spheres\\
\medskip
Connexions isospectrales, ergodicité des flots de repères et applications polynomiales entre sphères}

\author{Mihajlo Ceki\'c}
\address{Institut f\"ur Mathematik, Universit\"at Z\"urich, Winterthurerstrasse 190, CH-8057 Z\"urich, Switzerland}
\email{mihajlo.cekic@math.uzh.ch}

\author{Thibault Lefeuvre}
\address{Université de Paris and Sorbonne Université, CNRS, IMJ-PRG, F-75006 Paris, France.}
\email{tlefeuvre@imj-prg.fr}

\begin{document}

\dedicatory{Dedicated to the memory of Steve Zelditch}

\maketitle

\begin{abstract}
We show that on closed negatively curved Riemannian manifolds with simple length spectrum, the spectrum of the Bochner Laplacian determines both the isomorphism class of the vector bundle and the connection up to gauge under a low-rank assumption. We also show that flows of frames on low-rank frame bundles extending the geodesic flow in negative curvature are ergodic whenever the bundle admits no holonomy reduction. This is achieved by exhibiting a link between these problems and the classification of polynomial maps between spheres in real algebraic geometry.\\

Nous montrons que, sur les variétés riemanniennes fermées à courbure strictement négative et à spectre de longueur simple, le spectre du laplacien de Bochner détermine à la fois la classe d'isomorphisme du fibré vectoriel et la connexion à jauge près sous une hypothèse de rang faible. De plus, nous montrons que les flots partiellement hyperboliques obtenus comme extensions du flot géodésique à certains fibrés des repères de faible rang (sur des variétés à courbure strictement négative) sont ergodiques dès lors que le fibré n'admet aucune réduction d'holonomie. Ces résultats sont obtenus en établissant un lien entre ces problèmes et la classification des applications polynomiales entre sphères en géométrie \linebreak algébrique réelle.
\end{abstract}

\section{Introduction}

\label{section:intro}

The purpose of this article is to exhibit a link between the following three \emph{a priori} unrelated problems:
\begin{itemize}
\item In real algebraic geometry: the (non-)existence of polynomial mappings between spheres,
\item In dynamical systems: the ergodicity of certain partially hyperbolic dynamical systems obtained as isometric extensions of the geodesic flow over negatively curved Riemannian manifolds,
\item In spectral theory: Kac's isospectral problem ``\emph{Can one hear the shape of a drum}?'' for the Bochner (or connection) Laplacian.
\end{itemize}

\subsection{Background} In order to state the main theorems, we introduce some notation.

\subsubsection{Polynomial maps between spheres} 
Our main results will be formulated under a certain \emph{low-rank} assumption, which is governed by the quantity $q(n) \in \mathbb{Z}_{\geq 1}$ from real algebraic geometry. Namely, for $n \in \mathbb{Z}_{\geq 1}$, define $q(n)$ to be the \emph{least integer} $r \in \mathbb{Z}_{\geq 1}$ for which there exists a \emph{non-constant polynomial map} 
\[
	\mathbb{S}^n \to \mathbb{S}^r,
\] 
where $\mathbb{S}^n \subset \mathbb{R}^{n + 1}$ is the unit sphere; by polynomial map we mean the restriction to $\mathbb{S}^n$ of a polynomial map $\mathbb{R}^{n + 1} \to \mathbb{R}^{r + 1}$. It is known that $q(n)$ is even for $n \geq 2$, and a result due to Wood \cite{Wood-68} implies
\[
	n/2 < q(n) \leq  n, \qquad q(2^k) = 2^k, \quad k \in \mathbb{Z}_{\geq 0}.
\]
The first few values are given by:
\[
q(2) = q(3) = 2, \quad q(4) = \dotso = q(7) = 4, \quad q(8) = \dotso = q(15) = 8.
\]
For more properties of $q(n)$, see \S \ref{section:polynomial} below.

\subsubsection{Connections} Let $M$ be a smooth closed connected manifold. Let $\mathbf{A}^{\R}$ (resp.\ $\mathbf{A}^{\C}$) be the \emph{moduli space} of orthogonal connections on oriented Euclidean vector bundles (resp.\ unitary connections on Hermitian vector bundles) over $M$. Write $\mathbf{A}^{\F}_{r} \subset \mathbf{A}^{\F}$ (resp.\ $\mathbf{A}^{\F}_{\leq r} \subset \mathbf{A}^{\F}$) for the moduli space restricted to bundles of rank $r$ (resp.\ $\leq r$), where $\F = \R$ or $\C$. A point of $\mathbf{A}^\F$ corresponds to a pair $([E], [\nabla^E])$, where $[E] \to M$ is an equivalence class of vector bundles up to isomorphisms and $[\nabla^{E}]$ is an equivalence class of metric connections on $E$ up to \emph{gauge-equivalence}; we refer the reader to \cite[Chapter 2]{Donaldson-Kronheimer-90} for more details on connections and their moduli spaces.

\subsection{Isospectral connections}
Assume $M$ is endowed with a Riemannian metric $g$. Given $a = ([E], [\nabla^E]) \in \A^{\F}$, we can form the \emph{Bochner} (or \emph{connection}) \emph{Laplacian} $\Delta_E := (\nabla^{E})^* \nabla^E$. While this operator depends on the choice of $(E, \nabla^E)$, its unitary conjugacy class depends only on $a$; note also that $\Delta_E$ depends on the metric $g$ since the formal adjoint $(\nabla^{E})^*$ of $\nabla^{E}$ does so. It is a formally self-adjoint operator on $L^2(M,E)$ with discrete non-negative spectrum. Let 
\[
\mathrm{spec}_{L^2}(\Delta_E) := (\lambda_j(\nabla^E))_{j \geq 0}
\]
be the sequence of eigenvalues of $\Delta_E$, counted with multiplicities and sorted in increasing order; note that $\lambda_0(\nabla^E) \geq 0$.

Define the \emph{spectrum map} as
\begin{equation}
\label{equation:spectrum}
\mathbf{S} \colon \A^{\F} \ni a \mapsto \mathrm{spec}_{L^2}(\Delta_E) \in \R_{\geq 0}^{\Z_{\geq 0}}.
\end{equation}
The map $\mathbf{S}$ is well defined on the moduli space, that is, it does not depend on a choice of gauge for the connection. Following the celebrated paper of Kac \cite{Kac-66}, ``\emph{Can one hear the shape of a drum?}'', one can ask the following question: is the spectrum map \eqref{equation:spectrum} injective? In other words, does the spectrum of the Bochner Laplacian determine the connection up to gauge-equivalence? 

This is the analogous question to the classical inverse spectral problem of recovering a metric $g$ from the knowledge of the spectrum of the usual Hodge Laplacian $\Delta_g$ acting on functions. Among hyperbolic surfaces, it is known that the spectrum of the Hodge Laplacian does not determine the metric up to isometries by a result of Vigneras \cite{Vigneras-80}. Nevertheless, Sharafutdinov \cite{Sharafutdinov-09} proved that the spectrum map is locally injective in a neighbourhood of a locally symmetric Riemannian space of negative curvature. Apart from negatively curved spaces, other counterexamples were provided by Milnor \cite{Milnor-64}, Sunada \cite{Sunada-85} using covering spaces, and counterexamples to Kac's isospectral question also exist for piecewise smooth planar domains \cite{Gordon-Webb-Wolpert-92} (for the Dirichlet Laplacian). The problem of characterizing $1$-parameter families of isospectral metrics was also studied by various authors and was resolved in negative curvature \cite{Guillemin-Kazhdan-80, Croke-Sharafutdinov-98, Paternain-Salo-Uhlmann-14-2}. It is deeply connected to the \emph{marked length spectrum} conjecture (also known as the Burns-Katok conjecture \cite{Burns-Katok-85}), see \cite{Croke-90,Otal-90,Guillarmou-Lefeuvre-19}. We refer to \cite{Zelditch-04,Zelditch-14} for further details about Kac's classical isospectral problem for metrics.

Recall that a metric is said to have \emph{simple length spectrum} if all closed geodesics have different lengths. This is a generic condition with respect to the metric, see \cite{Abraham-70,Anosov-82}. Before stating our result on the injectivity of the spectrum map \eqref{equation:spectrum} on low-rank vector bundles, we set 
\[
	q_{\R}(n) := q(n), \quad q_{\C}(n) := \tfrac{1}{2}q(n).
\] 

\begin{theorem}
\label{theorem:main}
Let $n \in \mathbb{Z}_{\geq 2}$ and $\F = \R$ or $\C$. Then for all closed negatively curved Riemannian manifolds $(M^{n+1}, g)$ with simple length spectrum, the spectrum map is injective on $\F$-vector bundles of rank $\leq q_{\F}(n)$:
\[
	\mathbf{S} \colon \mathbf{A}^{\F}_{\leq q_{\F}(n)} \to \R_{\geq 0}^{\Z_{\geq 0}}.
\]
\end{theorem}

This seems to be the first result where Kac's inverse spectral problem can be fully solved with a full infinite-dimensional moduli space of geometric objects (here, connections on low-rank vector bundles). In particular, the situation is different from the metric case, where counterexamples are known to exist as discussed above. Earlier results obtained by the authors in \cite{Cekic-Lefeuvre-21-1} could only prove injectivity of the spectrum map \eqref{equation:spectrum} on small open neighborhoods in the moduli space $\mathbf{A}^{\F}$, but in arbitrary rank: for instance, on a neighborhood of a generic connection, or near flat connections. By means of a trace formula, this is intimately related to the \emph{holonomy inverse problem}, that is, the determination of a connection from the traces of parallel transports along closed geodesics. See Proposition \ref{proposition:2n} below for counterexamples to the injectivity of the holonomy inverse problem in even dimensions.

	Also note that, as a (very) particular instance of Theorem \ref{theorem:main}, one can take connections of the form $(M \times \C, d-iA)$ on the trivial complex line bundle, where $A \in C^\infty(M,T^*M)$ is a real-valued $1$-form; the Bochner Laplacian $(d-iA)^*(d-iA)$ is also known as the \emph{magnetic Laplacian} in this case. The result then says that one can recover the $1$-form $A$ up to gauge (i.e. up to a term of the form $i\varphi^{-1} d\varphi$ for some $\varphi \in C^\infty(M,\Ss^1)$) from the knowledge of the spectrum of the magnetic Laplacian. In particular, this allows to recover the magnetic field $B := dA$ from the spectrum.

In the general case, counterexamples to the injectivity of the spectrum map \eqref{equation:spectrum} were constructed by Kuwabara \cite{Kuwabara-90} using the Sunada method \cite{Sunada-85} but on coverings of a given Riemannian manifold (the simple length spectrum condition is thus violated). As we shall see in \S\ref{section:spectral}, it is natural to conjecture the following: 

\begin{conjecture}
The spectrum map \eqref{equation:spectrum} is injective on odd-dimensional closed Riemannian manifolds $(M^{2n+1},g)$ with negative sectional curvature and simple length spectrum.
\end{conjecture}

However, it is not clear yet at all how to get rid of the low-rank assumption in Theorem \ref{theorem:main} on odd-dimensional manifolds.

\subsection{Ergodicity of frame flows} Using similar techniques as the ones involved in the proof of Theorem \ref{theorem:main}, we can also prove results concerning the ergodicity of a class of dynamical systems.

Let $(M^{n+1},g)$ be a smooth closed connected Riemannian manifold with negative sectional curvature and let $\pi \colon SM \to M$ be its unit tangent bundle. Given $a = ([E],[\nabla^E]) \in \A_r^{\F}$, consider the pullback pair
\[
\pi^*a = \left([\pi^*E],[\pi^*\nabla^E]\right) =: ([\E], [\nabla^{\E}])
\]
of a vector bundle and a connection over $SM$. Let $(\varphi_t)_{t \in \R}$ be the geodesic flow on $SM$ and define $F\E$ to be the principal $G$-bundle (with $G = \mathrm{SO}(r)$ or $\mathrm{U}(r)$) of orthonormal frames of $\E$. Define the \emph{frame flow} $\Phi_t \colon F\E \to F\E$ of $\nabla^{\E}$ by
\begin{equation}\label{eq:frame-flow}
	\Phi_t(x, v, \e_1, \dotsc, \e_r) := (\varphi_t(x, v), \mc{P}_{\gamma(t)} \e_1, \dotsc, \mc{P}_{\gamma(t)} \e_r),
\end{equation}
where $(\e_1, \dotsc, \e_r)$ is an orthonormal frame of $E_x$ and $\mc{P}_{\gamma(t)}$ denotes the parallel transport of $\nabla^{E}$ along the geodesic segment $\gamma(s) = \pi \circ \varphi_s(x, v)$ for $s \in [0, t]$. This flow extends $\varphi_t$ in the sense that $\varphi_t \circ p = p \circ \Phi_t$, where $p \colon F\E \to SM$ is the projection, and it commutes with the right-action of the group $G$. The geodesic flow preserves a natural smooth measure called the \emph{Liouville measure}, and $(\Phi_t)_{t \in \R}$ thus preserves a canonical measure $\mu$ on $F\E$ obtained locally as the product of the Liouville measure with the Haar measure on the group $G$. Note that the flows $(\Phi_t)_{t \in \mathbb{R}}$ obtained for connections in the class $([E], [\nabla^E])$ are smoothly conjugate to each other by volume-preserving maps and hence may be identified.

Our aim is to study the ergodicity of $(\Phi_t)_{t \in \R}$ with respect to $\mu$. Such flows are typical examples of \emph{partially hyperbolic dynamical systems} which are still a vast topic of investigation, see \cite{Hasselblatt-Pesin-06} for an introduction for instance. A standard example of such a flow is given by the \emph{geodesic frame flow} over negatively curved manifolds, defined as in \eqref{eq:frame-flow} by taking frames $(\e_1, \dotsc, \e_{n})$ of $v^\perp \subset T_xM$ and parallelly transporting them with respect to the Levi-Civita connection $\nabla^{\mathrm{LC}}$. Note that this is \emph{not} a special case of the above construction for $E = TM$ and $\nabla^{E} = \nabla^{\mathrm{LC}}$, since instead of the pullback bundle $\pi^*TM$ for this flow we are using its subbundle $\E := s^\perp$ where $s(x, v) = v$ is the tautological section of $\pi^*TM$.

It was proved by Brin \cite{Brin-75-1} and Brin--Gromov \cite{Brin-Gromov-80} that the geodesic frame flow is ergodic on all odd-dimensional negatively curved manifolds of dimension different from $7$. More recently, it was proved by the authors together with Moroianu--Semmelmann \cite{Cekic-Lefeuvre-Moroianu-Semmelmann-21} that this flow is ergodic on even-dimensional manifolds when the manifold has nearly $0.25$-pinched (resp.\ $0.5$-pinched) negative sectional curvature and the dimension is not divisible by $4$ (resp.\ is divisible by $4$). Similar results were also obtained in \cite{Cekic-Lefeuvre-Moroianu-Semmelmann-21} in dimension $7$. This almost answered a conjecture of Brin \cite[Conjecture 2.6]{Brin-82}. Previous partial results can be found in \cite{Brin-Karcher-83, Burns-Pollicott-03}.

Given a pair $([E], [\nabla^{E}]) \in \mathbf{A}_r^{\F}$ over $M$, say that it admits a \emph{holonomy reduction} if the (full) \emph{holonomy group} $\mathrm{Hol}(E,\nabla^E)$ is strictly contained in $\mathrm{SO}(r)$ (when $\F = \mathbb{R}$) or $\mathrm{U}(r)$ (when $\F = \mathbb{C}$). In particular, taking $E = TM$ and $\nabla^E = \nabla^{\mathrm{LC}}$, one recovers the usual \emph{Riemannian} holonomy group $\mathrm{Hol}(M) := \mathrm{Hol}(TM,\nabla^{\mathrm{LC}})$. It will become clear that the ergodicity of the frame flow implies that there is no holonomy reduction (see Theorem \ref{prop:brin} below).

Before stating our result, we let
\begin{equation*}
m_{\R}(n) := \Big\lfloor{\tfrac{1}{2}\Big(1+\sqrt{1+8q_{\R}(n)}\Big)}\Big\rfloor, \qquad m_{\C}(n) :=  \Big\lfloor{\sqrt{q_\R(n)}}\Big\rfloor.
\end{equation*}
In the real case, the initial values are:
\[
\begin{split}
	m_{\R}(2) = m_{\R}(3)=2,\quad m_{\R}(4) = \dotso = m_{\R}(7) = 3,\quad m_{\R}(8) = \dotso = m_{\R}(15) = 4,
\end{split}
\]
while in the complex case, we get:
\[
\begin{split}
	m_{\C}(2) = m_{\C}(3) = 1,\quad m_{\C}(4) = \dotso = m_{\C}(7) = 2,\quad m_{\C}(8) = \dotso = m_{\C}(15) = 2.
\end{split}
\]

\begin{theorem}
\label{theorem:ergodicity}
Let $n \in \mathbb{Z}_{\geq 2}$ and $\mathbb{F} = \mathbb{R}$ or $\mathbb{C}$, and let $(M^{n+1},g)$ be a closed negatively curved Riemannian manifold. Assume $([E], [\nabla^E]) \in \A^{\F}_{\leq m_{\F}(n)}$ has no holonomy reduction and set $([\E], [\nabla^{\E}]) := ([\pi^*E], [\pi^*\nabla^{E}])$. Then the frame flow $(\Phi_t)_{t \in \R}$ on $F\E$ is ergodic.
\end{theorem}

The case $\dim(M)=2$ will be discussed afterwards in \S\ref{section:ergodicity}, see Corollary \ref{corollary:surfaces}. It can be easily seen that the absence of holonomy reduction is a generic property among connections. Let us point out that in \cite[Conjecture 2.9]{Brin-82}, Brin conjectured that the geodesic frame flow over negatively curved manifolds should be ergodic whenever the Riemannian manifold $(M,g)$ admits no holonomy reduction. It follows from Berger's classification \cite{Berger-53} that negatively curved manifolds with strictly $0.25$-pinched curvature have no Riemannian holonomy reduction, namely, $\mathrm{Hol}(M)=\mathrm{SO}(n + 1)$. Theorem \ref{theorem:ergodicity} therefore proves a general version of Brin's conjecture for arbitrary flows of frames (on pullback vector bundles) but under a low-rank assumption on the bundle. Unfortunately, the usual geodesic frame flow fails to satisfy the assumptions of Theorem \ref{theorem:ergodicity} because: (i) the rank of the tangent bundle is too high; (ii) as explained above, it does not fit directly into the pullback bundle framework. We refer to \cite{Cekic-Lefeuvre-Moroianu-Semmelmann-21,Cekic-Lefeuvre-Moroianu-Semmelmann-22} for more details on the geodesic frame flow.

\subsection{Proof ideas.} The ideas of proofs of Theorems \ref{theorem:main} and \ref{theorem:ergodicity} are similar; we outline that of the former. 
 
Given two isospectral connections $a_1 = ([E_1],[\nabla^{E_1}])$ and $a_2 = ([E_2], [\nabla^{E_2}])$, by \emph{non-Abelian Liv\v{s}ic theory} we will construct an isometry 
\[
	p \in C^\infty(SM,\mathrm{Hom}(\pi^*E_1,\pi^*E_2))
\] 
which is invariant by parallel transport along geodesic flowlines with respect to the natural connection induced on $\mathrm{Hom}(\pi^*E_1,\pi^*E_2)$. Any function (or section of a pullback vector bundle) defined on $SM$ can be decomposed on each spherical fiber into a (potentially infinite) sum of spherical harmonics. Using a result due to Guillarmou-Paternain-Salo-Uhlmann (see \cite[Theorem 4.1]{Guillarmou-Paternain-Salo-Uhlmann-16}), we shall see below that invariant sections, such as the section $p$ defined above, must necessarily have \emph{finite Fourier degree}. That is, their decomposition into a sum of spherical harmonics is always finite, see Proposition \ref{theorem:finite} below. The proof of \cite[Theorem 4.1]{Guillarmou-Paternain-Salo-Uhlmann-16} is based on the twisted Pestov identity, which can be viewed as a Weitzenböck-type identity in Riemannian geometry, as subsequently discussed in \cite{Cekic-Lefeuvre-Moroianu-Semmelmann-23}.

In turn, as spherical harmonics are restrictions of homogeneous polynomials to the unit sphere, this means that such sections are fiberwise \emph{polynomial}. As all the bundles involved are pullback bundles to $SM$ (hence, trivial over the sphere fibers of the unit tangent bundle), we will get that, by simply restricting to a single sphere $S_{x_0}M \simeq \Ss^n$ over a point $x_0 \in M$, these invariant sections provide polynomial maps $\Ss^n \to \mathrm{SO}(r)$ or $\mathrm{U}(r)$. Under the low-rank assumption, these maps are known not to exist (see \S \ref{section:polynomial}) except for the trivial constant map, by real algebraic geometry. Having such a fiberwise constant section on $SM$ is the same as having that the section is a pullback section from the base $M$, that is $p = \pi^* p_0$ for some $p_0 \in  C^\infty(M,\mathrm{Hom}(E_1,E_2))$ parallel with respect to the natural connection. This allows us to conclude that the connections are gauge-equivalent.\\

\noindent \textbf{Organisation of the article:} In \S\ref{section:polynomial}, we study basic properties of polynomial maps between spheres. In \S\ref{section:ergodicity} and \S\ref{section:spectral}, we apply this to show ergodicity of frame flows in low-rank (Theorem \ref{theorem:ergodicity}) and global uniqueness for the inverse spectral problem of the Bochner Laplacian (Theorems \ref{theorem:main} and \ref{theorem:wilson}), respectively. \\

\noindent \textbf{Acknowledgement:} We thank A.\ Moroianu and U.\ Semmelmann for helpful discussions, especially for the proof of Proposition \ref{proposition:2n}. We also thank O. Benoist for pointing out the references \cite{Yiu-94,Totaro-07}, and O. Randal-Williams for answering our questions about K-theory. Finally, we warmly thank the referees and the editors for their valuable suggestions improving the presentation. M.C. has received funding from an Ambizione grant (project number 201806) from the Swiss National Science Foundation.

\section{Polynomial maps between spheres}

\label{section:polynomial}

In this section, we summarise basic properties of polynomial maps between spheres and derive consequences for polynomial maps from spheres to some other spaces.

\subsection{Polynomial maps}

\label{ssection:poly}

Let $n \geq 1$ be an integer and $\Ss^n \subset \R^{n+1}$ denote the usual round sphere. We are interested in polynomial maps $\Ss^n \to G(r)$, with $r \in \Z_{\geq 0}$ and
\begin{equation}\label{eq:G-family}
G(r) = \Ss^r, \mathrm{SO}(r), \mathrm{U}(r), \mathrm{SU}(r), \mathrm{Gr}_{\F}(k,r),
\end{equation}
where $\mathrm{Gr}_{\F}(k,r)$ is the Grassmannian of $k$-planes in $\F^r$ for some $0 \leq k \leq r$ with $\F = \R$ or $\C$. By polynomial, we mean that the coordinates of $F \colon \Ss^n \to G(r)$ are polynomials in the variable $v \in \R^{n+1}$, that is, the map is the restriction to $\Ss^n$ of a finite collection of polynomials on $\R^{n+1}$ (not necessarily homogeneous). For instance, a matrix map $\Ss^n \to \mathrm{SO}(r)$ is polynomial if all of its entries are polynomial.

For $G$ one of the families in \eqref{eq:G-family} let $q_G(n)$ be the \emph{least integer} $r \in \Z_{\geq 0}$ for which there exists a \emph{non-constant} polynomial map $\Ss^n \to G(r)$. Write simply $q(n) := q_{\Ss}(n)$ in the case of the sphere. We shall see below that $q_{G}(n)$ are actually explicitly determined by $q(n)$.

Using the identity map $\Ss^n\to \Ss^n$, we see that $q(n)\leq n$, and the problem is thus about finding non-constant maps $\Ss^n\to \Ss^r$ for $r<n$. The exact value of $q(n)$ for general $n$ remains an open question but an important result was obtained by Wood \cite{Wood-68}:

\begin{theorem}[Wood \cite{Wood-68}]
\label{theorem:wood}
Let $n$ and $r$ be integers such that $1 \leq r \leq n-1$. Assume that there exists a power of $2$ among $\left\{r+1, \dotsc,n\right\}$. Then, there exists no non-constant polynomial mapping $\Ss^n \to \Ss^r$.
\end{theorem}

From Theorem \ref{theorem:wood}, we have the following bounds:
\begin{equation}
\label{equation:q}
n/2 < q(n) \leq  n, \qquad q(2^k) = 2^k, \quad k \in \mathbb{Z}_{\geq 0}.
\end{equation} 
We point out that the classification of quadratic polynomial mappings was completely settled by Yiu \cite{Yiu-94}. In particular, this provides an \emph{explicit} upper bound $q(n) \leq q_2(n)$, where $q_2(n)$ is the least integer $r$ such that there exists a quadratic polynomial mapping $\Ss^n \to \Ss^r$ (see \cite[Theorem 4]{Yiu-94} for the precise value of $q_2(n)$). But the general explicit determination of $q(n)$ seems out of reach for the moment. Using the three Hopf fibrations $\Ss^3 \to \Ss^2, \Ss^7 \to \Ss^4$, and $\Ss^{15} \to \Ss^8$ (which are quadratic polynomial maps), the initial values of $q(n)$ can be easily computed and one gets
\begin{equation}\label{eq:15}
q(2)=q(3)=2, \quad q(4) = \dotso = q(7) = 4, \quad q(8) = \dotso = q(15) = 8.
\end{equation}
Although less obvious, using the \emph{Hopf construction} (see \S \ref{ssec:hopf}) it can also be proved that 
\begin{equation}
\label{equation:16}
q(16) = \dotso = q(31) = 16, \qquad q(32) = \dotso = q(47) = 32.
\end{equation}
The first unknown value seems to be $q(48)$, that is, it is not known whether there exists a polynomial mapping $\Ss^{48} \to \Ss^{47}$ (necessarily of degree $\geq 3$ by \cite{Yiu-94}), see \cite[top of page 6]{Totaro-07} where this is discussed.

\subsection{Hopf construction.}\label{ssec:hopf} First of all, let us introduce one recipe for defining polynomial maps that we call the \emph{Hopf construction}.  Namely, given a bilinear map $F \colon \R^r \times \R^s \to \R^t$ such that $|F(x,y)|^2 = |x|^2|y|^2$ for any $(x, y) \in \R^r \times \R^s$, one defines 
\begin{equation}\label{eq:hopf-construction}
	H \colon \R^r \times \R^s \to \R^{t+1}, \quad H(x,y) := (|x|^2-|y|^2,2F(x,y)),
\end{equation}
which yields a natural quadratic map $\Ss^{r+s-1} \to \Ss^t$ as $|H(x,y)|=1$ for $|x|^2+|y|^2=1$. 

For any fixed odd integer $n$, let $\rho(n+1)$ be the Radon-Hurwitz number; then it is known that $\rho(n+1)-1$ is the maximal number of linearly independent vector fields on $\Ss^n$ (see \cite[Section I.7]{Lawson-Michelsohn-89}). We can thus construct a Hopf map $\Ss^{n+\rho(n+1)} \to \Ss^{n+1}$ by taking $F \colon \R^{n+1} \times \R^{\rho(n+1)} \to \R^{n+1}$ such that
\begin{equation}\label{eq:def-normed-bilinear-F}
F(x,y) = y_0 x + y_1 J_1x + \dotsb + y_{\rho(n+1)-1} J_{\rho(n+1)-1} x,
\end{equation}
where $J_1, \dotsc ,J_{\rho(n+1)-1}$ are orthogonal almost-complex structures on $\R^{n+1}$ (this corresponds to the representation of the Clifford algebra $\mathrm{C}\ell_{\rho(n+1)-1}$ on $\R^{n+1}$, see \cite[Section I.7]{Lawson-Michelsohn-89}). Of course, taking $n=1,3,7$, one recovers the usual three Hopf fibrations $\Ss^3 \to \Ss^2, \Ss^7 \to \Ss^4$, and $\Ss^{15} \to \Ss^8$.

By the Hopf construction there is a quadratic map $\Ss^{31} \to \Ss^{24}$ and another quadratic map $\Ss^{24} \to \Ss^{16}$, thus giving a map $\Ss^{31} \to \Ss^{16}$ of polynomial degree $4$. Combining with Lemma \ref{lemma:non-decreasing} below, this proves the first part of \eqref{equation:16}. One similarly argues to obtain the second part of \eqref{equation:16} using the composition $\Ss^{47} \to \Ss^{40} \hookrightarrow \Ss^{41} \to \Ss^{32}$, where the first and last arrows are obtained by the Hopf construction and the middle one is inclusion.

\subsection{Properties of $q_G(n)$.}
The following holds:
 
 \begin{lemma}
 \label{lemma:non-decreasing}
 The map $n \mapsto q(n)$ is non-decreasing.
 \end{lemma}

 \begin{proof}
 Assume $n \geq 2$ and let $F: \mathbb{S}^{n} \to \mathbb{S}^{q(n)}$ be a non-constant polynomial map. Then, there exists a great sphere of $\mathbb{S}^{n}$ to which $F$ restricts as a non-constant polynomial map (with degree at most that of $F$); this shows $q(n - 1) \leq q(n)$ and concludes the proof.
 \end{proof}

\begin{lemma}
\label{lemma:even}
For all $n \in \mathbb{Z}_{\geq 2}$, $q(n)$ is even.
\end{lemma}

\begin{proof}
The Hopf construction \eqref{eq:hopf-construction} always provides a quadratic mapping $\Ss^{2k+1} \to \Ss^{2k}$ by taking $r=2k$, $s =2$, $t=2k$, $J \in \Lambda^2 \R^{2k}$, an almost-complex structure and by setting for $x \in \R^{2k}$, $y \in \R^2$, $F(x,y) = y_1x + y_2Jx$, similarly to \eqref{eq:def-normed-bilinear-F}. 

Now, we claim that if there is a non-constant polynomial mapping $F \colon \Ss^n \to \Ss^{2k+1}$, then the composition with the Hopf construction still provides a non-constant polynomial mapping $\Ss^n \to \Ss^{2k+1} \to \Ss^{2k}$. There are two cases: if $F(\Ss^n) \subset \Ss^{2k+1}$ is contained in a great sphere of $\Ss^{2k+1}$, then we actually have a non-constant polynomial map $F \colon \Ss^n \to \Ss^{2k}$. If not, it suffices to use that the preimage of a point $v \in \Ss^{2k}$ under the Hopf mapping $H\colon \Ss^{2k+1} \to \Ss^{2k}$ is the intersection of $\Ss^{2k+1}$ with a linear subspace of $\R^{2k+2}$ (in particular, it is contained in a great sphere), see \cite[Theorem 1.4]{Yiu-86}.
\end{proof}

For every family $G$ in \eqref{eq:G-family}, we are able to determine $q_{G}(n)$ from $q(n)$.

\begin{lemma}
\label{lemma:determine}
Assume that $n \geq 2$. Then, one has: 
\[
\begin{split}
&q_{\mathrm{SO}}(n) = 1+q(n) , \qquad q_{\mathrm{U}}(n) = q_{\mathrm{SU}}(n) = 1 + \tfrac{1}{2}q(n).
\end{split}
\]
Moreover, we have
\[q_{\mathrm{Gr}_{\R}(k, \bullet)}(n) = 1+\max\big(k, q(n)\big) , \qquad q_{\mathrm{Gr}_{\C}(k, \bullet)}(n) = 1 + \max \big(k, \tfrac{1}{2}q(n)\big).\]
\end{lemma}

\begin{proof}
We start with the following observation: for all $r \geq 1$, there are natural polynomial (quadratic) mappings
\begin{equation}
\label{equation:natural}
\Ss^{r} \to \mathrm{SO}(r + 1), \qquad \Ss^{2r + 1} \to \mathrm{SU}(r + 1),
\end{equation}
obtained (up to sign) by taking the reflection with respect to the $\F$-span of $v$, that is, $v \mapsto (-1)^r\big(2\pi_v - \mathbbm{1}_{\F^{r + 1}}\big)$, where $\pi_v(\bullet) := \langle v, \bullet \rangle v$ is the orthogonal projection onto the $\F$-line spanned by $v$, $\langle \bullet,\bullet\rangle$ is the standard Euclidean or Hermitian metric on $\F^{r + 1}$. (Note that there is also a quadratic map $\Ss^{4r + 3} \to \mathrm{Sp}(r + 1)$ in the quaternionic case. The sign $ (-1)^{r}$ is chosen such that the determinant is equal to $1$.)

We first deal with the real case. Assume that we have a non-constant polynomial map $F \colon \Ss^n \to \Ss^{q(n)}$. Then, we can use \eqref{equation:natural} to produce a polynomial map $\Ss^n \to \mathrm{SO}(q(n) + 1)$ given by $v \mapsto (-1)^{q(n)}\big(2\pi_{F(v)} - \mathbbm{1}_{\R^{q(n) + 1}}\big)$. We claim that this is non-constant: indeed, if it were constant, then $\pi_{F(v)}(\bullet) = \langle \bullet, F(v)\rangle F(v)$ would be constant which is absurd. This gives $q_{\mathrm{SO}}(n) \leq q(n)+1$. On the other hand, assume that we have a non-constant polynomial map $F \colon \Ss^n \to  \mathrm{SO}(q_{\mathrm{SO}}(n))$. Then, we can cook up a polynomial map $\Ss^n \to \mathrm{SO}(q_{\mathrm{SO}}(n)) \to \Ss^{q_{\mathrm{SO}}(n)-1}$, where the second arrow is obtained by taking one of the columns of the matrix. This map can be chosen to be non-constant, otherwise the whole matrix would be constant, which contradicts the assumption. This gives $q_{\mathrm{SO}}(n)-1 \geq q(n)$, and thus $q_{\mathrm{SO}}(n) =1+q(n)$. 

Next, if $k \leq q(n)$, the analogous argument works for $\mathrm{Gr}_{\R}(k, \bullet)$ since there are maps $\mathrm{SO}(r) \to \mathrm{Gr}_{\R}(k,r)$ and $\mathrm{Gr}_{\R}(k,r) \to \mathrm{SO}(r)$ (as long as $1 \leq k \leq r-1$) obtained by taking respectively the projection onto the linear span of a suitable set of $k$ columns of the matrix, and the orthogonal reflection with respect to the $k$-plane. If $k > q(n)$, then we get a non-constant mapping trivially as follows, where the middle and the last maps are equatorial inclusion and orthogonal projection onto $v^\perp$, respectively:
\[\mathbb{S}^n \overset{F}{\to} \mathbb{S}^{q(n)} \hookrightarrow \mathbb{S}^k \to \mathrm{Gr}_{\R}(k, k+1).\]
This completes the proof for Grassmanians over $\mathbb{R}$. 

Let us now deal with the complex case. Observe that, given a polynomial map $p \colon \Ss^n \to \mathrm{U}(r)$, $\det p : \Ss^n \to \mathrm{U}(1) \simeq \Ss^1$ is also polynomial, hence constant by Theorem \ref{theorem:wood}. Thus, up to rescaling by a constant factor, a polynomial map $\Ss^n \to \mathrm{U}(r)$ is the same as a polynomial map $\Ss^n \to \mathrm{SU}(r)$, which gives $q_{\mathrm{U}}(n)=q_{\mathrm{SU}}(n)$. Next, observe that there are polynomial mappings
\[
\Ss^n \to \mathrm{U}(q_{\mathrm{U}}(n)) \hookrightarrow \mathrm{SO}(2q_{\mathrm{U}}(n)) \to \Ss^{2q_{\mathrm{U}}(n)-1},
\]
where the first arrow is a non-constant polynomial map and the last picks out one of the columns so that the composition is non-constant. This shows $2q_{\mathrm{U}}(n)-1 \geq 2q_{\mathrm{U}}(n)-2 \geq q(n)$ since $q(n)$ is even by Lemma \ref{lemma:even}. Moreover, we also have
\[
\Ss^n \to \Ss^{q(n)} \hookrightarrow \Ss^{q(n) + 1} \to \mathrm{U}\big(\tfrac{1}{2}q(n) + 1\big),
\]
where the second arrow is equatorial inclusion, and the last arrow is the natural mapping \eqref{equation:natural} sending $v$ to the unitary reflection with respect to the $\C$-span of $v$. This gives $q(n)=2q_{\mathrm{U}}(n)-2$.

Finally, the analogous argument as in the real case gives the formula for $q_{\mathrm{Gr}_{\C}(k)}(n)$. 
\end{proof}

As we shall see below in \S\ref{section:ergodicity} and \S\ref{section:spectral}, polynomial maps $\Ss^n \to G(r)$ will naturally appear as maps whose coordinates are a finite sum of spherical harmonics. Recall that the space of $L^2$ functions on $\Ss^n$ breaks up as
\begin{equation}
\label{equation:l2}
L^2(\Ss^n) = \oplus_{k \geq 0} \Omega_k(\Ss^n), \qquad \Omega_k(\Ss^n) := \ker(\Delta_{\Ss^n}-k(k+n-1)),
\end{equation}
where $\Delta_{\Ss^n}$ is the Laplacian induced by the round metric on the sphere. Writing $\mathbf{H}_\ell(\R^{n+1})$ for the space of homogeneous harmonic polynomials of degree $\ell \geq 0$ in $\R^{n+1}$, the restriction map
\[
r \colon C^\infty(\R^{n+1}) \ni f \mapsto f|_{\Ss^{n}} \in C^\infty(\Ss^n),
\]
yields an isomorphism
\begin{equation}
\label{equation:isomorphisme}
r \colon \mathbf{H}_{\ell}(\R^{n+1}) \to \Omega_{\ell}(\Ss^n).
\end{equation}
Hence, there is a natural identification 
\begin{equation}
\label{equation:identification}
\oplus_{k=0}^{\ell} \Omega_k(\Ss^n) \simeq \oplus_{k=0}^{\ell} \mathbf{H}_{k}(\R^{n+1}) \subset \R_{\ell}[v],
\end{equation}
where $\R_\ell[v]$ denotes the space of polynomials in the $v \in \R^{n+1}$ variable of degree $\leq \ell$. We will say that a function $f \in C^\infty(\Ss^n)$ has \emph{finite Fourier content} if its $L^2$-decomposition \eqref{equation:l2} is a finite sum of spherical harmonics. By \eqref{equation:identification}, a map $\Ss^n \to G(r)$ whose coordinates have finite Fourier content is in particular a polynomial map.

\section{Ergodicity of isometric extensions of the geodesic flow}

\label{section:ergodicity}

Let $(M^{n+1},g)$ be a smooth closed Riemannian manifold. From now on, we will always assume that the geodesic flow $(\varphi_t)_{t \in \R}$ defined on the unit tangent bundle 
\[
SM := \left\{ (x,v) ~|~ |v|_g=1\right\}
\]
is \emph{Anosov}, in the sense that its tangent bundle $T(SM)$ splits as
\[
T(SM)=\R X \oplus E^u \oplus E^s,
\]
where $X$ is the geodesic vector field, and for some $C,\lambda > 0$, for all $t \geq 0$
\begin{equation}
\label{equation:anosov}
 \|d\varphi_{t} v\| \leq C e^{-\lambda t} \|v\|,\ \forall v \in E^{s},  \qquad  \|d\varphi_{- t} v\| \leq C e^{-\lambda t} \|v\|,\ \forall v \in E^{u},
\end{equation}
where $\|\bullet\|$ is an arbitrary norm on $SM$. Typical examples are provided by negatively curved manifolds \cite{Anosov-67}. In this section, we investigate certain \emph{isometric extensions} of the geodesic flow (see \cite{Lefeuvre-23} for a general discussion of those). Here, we will be only interested in specific extensions coming from parallel transport on vector bundles.

\begin{definition}
\label{definition:reduction}
Given a pair $(E,\nabla^{E}) \in \mathbf{A}^{\R}_r$ over $M$, we say that it admits a \emph{holonomy reduction} if the full holonomy group $\mathrm{Hol}(E,\nabla^E) \leqslant \mathrm{SO}(r)$ is strictly contained in $\mathrm{SO}(r)$. We say that it is \emph{irreducible} if there exists no non-trivial splitting $(E,\nabla^{E}) = (E_1,\nabla^{E_1}) \oplus (E_2,\nabla^{E_2})$ over $M$. The same definitions hold in the complex case with the obvious modifications.
\end{definition}

As in \S\ref{section:intro}, for $a = ([E],[\nabla^E])$, we can consider
\[
\pi^*a = \left([\pi^*E],[\pi^*\nabla^E]\right) =: ([\E], [\nabla^{\E}])
\]
where $\pi \colon SM \to M$ is the footpoint projection. Recall that the \emph{frame flow} $\Phi_t \colon F\E \to F\E$ is defined in \eqref{eq:frame-flow} as the parallel transport with respect to $\nabla^{\E}$ of a frame $w \in F\E$ over the point $(x,v) \in SM$ along the geodesic $(\varphi_t(x,v))_{t \in \R}$. The geodesic flow preserves the Liouville measure, and $(\Phi_t)_{t \in \R}$ thus preserves a canonical measure $\mu$ on $F\E$ obtained locally as the product of the Liouville measure with the Haar measure on the group $G$ (with $G = \mathrm{SO}(r)$ in the real case, and $G = \mathrm{U}(r)$ in the complex case). In this section, we investigate the ergodicity of $(\Phi_t)_{t \in \R}$ with respect to $\mu$.

Following \cite{Cekic-Lefeuvre-21-1,Lefeuvre-23}, it is convenient for the purpose of studying ergodicity to introduce \emph{Parry's free monoid} and the \emph{transitivity group} of the extension $(\Phi_t)_{t \in \R}$. For what follows, we can assume that $\M$ (set to be $SM$ in the geodesic case) is any closed manifold carrying a transitive Anosov flow $(\varphi_t)_{t \in \R}$ as in \eqref{equation:anosov} generated by a vector field $X \in C^\infty(\M,T\M)$, $\E \to \M$ is a real or complex vector bundle equipped with an orthogonal or unitary connection $\nabla^{\E}$, respectively, $\X := \nabla_X^{\E}$, and $(\Phi_t)_{t \in \R}$ is the frame flow on $F\E$ as defined previously. Let $X_{F\E}$ be the vector field on $F\E$ generating $(\Phi_t)_{t \in \R}$. Note that there is also a natural flow (still denoted by $(\Phi_t)_{t \in \R}$) on the unit sphere bundle of $\E$, $S\E := \{(x, v, e) \in \E \mid |e|_x = 1\}$, where $|\bullet|_x$ denotes the norm in the fibre $E_x$, by simply considering parallel transport of a single section of $\E$ (and not a whole frame). Let us fix an arbitrary periodic point $z_\star \in \M$ of $\varphi_t$ and denote by $\gamma_\star$ its orbit.

Introduce $\mc{H}$, the set of \emph{homoclinic} orbits to $\gamma_\star$ (the orbits accumulating to $\gamma_\star$ in the past and in the future) and $\mathbf{G}$, \emph{Parry's free monoid}, defined as the following formal free monoid:
\[
\mathbf{G} := \left\{\gamma_{1}^{k_1} \dotsm \gamma_{p}^{k_p} ~|~ \gamma_{1}, \dotsc,\gamma_p \in \mc{H}, k_j \in \Z_{\geq 0}\right\}.
\]
Note that $\gamma_\star \in \mathbf{G}$. Given a connection, there is a natural way to define parallel transport from $
\E_{z_\star} \to \E_{z_\star}$ along each of these homoclinic orbits $\gamma \in \mc{H}$ (even though they have infinite length) using the notion of \emph{stable/unstable holonomies}.

This is defined as follows: given $\gamma \in \mc{H}$ with $\gamma \neq \gamma_\star$, pick arbitrary points $z_- \in W^u_{\M}(z_\star) \cap \gamma$ and $z_+\in W^s_{\M}(z_\star) \cap \gamma$, where $W^{s/u}_{\M}$ denote the strong stable/unstable manifolds associated to the Anosov flow $(\varphi_t)_{t \in \R}$ on $\M$ (see e.g. \cite[Chapter 6]{Fisher-Hasselblatt-19}). Consider an arbitrary Riemannian metric on $\M$, and for $x, y \in \M$ sufficiently close denote by $P_{x \to y}: \E_x \to \E_y$ the parallel transport with respect to $\nabla^{\E}$ along the shortest geodesic between $x$ and $y$. Then define the (linear) unstable holonomy $\mathrm{Hol}^u_{z_\star \to z_-} \colon \E_{z_\star} \to \E_{z_-}$ (and similarly $\mathrm{Hol}^s_{z_+ \to z_\star} \colon \E_{z_+} \to \E_{z_\star}$) by setting for $f \in \E_{z_\star}$:
\begin{equation}
\label{equation:unstable-holonomy}
	\mathrm{Hol}^u_{z_\star \to z_-} f := \lim_{t \to \infty} \Phi_t \circ P_{\varphi_{-t}z_\star \to \varphi_{-t}z_-} \circ \Phi_{-t} f,
\end{equation}
which converges thanks to the Ambrose-Singer formula \cite[Lemma 3.14]{Cekic-Lefeuvre-21-1}. It is clear from \eqref{equation:unstable-holonomy} that the unstable holonomy is a linear map; a similar definition to \eqref{equation:unstable-holonomy} works for the stable holonomy. Since $z_+ = \varphi_T(z_-)$ for some time $T > 0$, we can then define
\begin{equation}
\label{equation:holonomy}
\rho(\gamma) f := \mathrm{Hol}^s_{z_+ \to z_\star} \circ \Phi_T \circ \mathrm{Hol}^u_{z_\star \to z_-} f.
\end{equation}
We emphasize here that $\rho(\gamma)$ \emph{does} depend on the choice of points $z_\pm$. Making a different choice of points $z'_\pm$ instead of $z_\pm$ as above leads to a different $\rho'(\gamma)$ related to $\rho(\gamma)$ by
\begin{equation}
\label{equation:co}
\rho'(\gamma)=\rho(\gamma_\star)^a \rho(\gamma) \rho(\gamma_\star)^b,
\end{equation}
where $a,b \in \Z$. We refer to \cite[Proposition 3.15]{Cekic-Lefeuvre-21-1} or \cite[Section 2.1]{Lefeuvre-23} for further details on this construction. 

Identifying isometrically $\E_{z_\star} \simeq \F^r$, we thus get by \eqref{equation:holonomy} a monoid representation
\begin{equation}
\label{equation:parry-representation}
\rho \colon \mathbf{G} \to G,
\end{equation}
where $G = \mathrm{SO}(r)$ or $\mathrm{U}(r)$. The closure of the image of this representation $H := \overline{\rho(\mathbf{G})}$ is called the \emph{transitivity group}. It is \emph{independent} of the choice of points $z_\pm$ in the previous definition of the holonomy, as can be seen from the relation \eqref{equation:co}. It plays an important role as it describes the ergodic components of the frame flow $(\Phi_t)_{t \in \R}$ on $F\E$. This was discovered by Brin \cite{Brin-75-1,Brin-75-2} and refined later in \cite{Lefeuvre-23}:

\begin{theorem}[Brin '75]\label{prop:brin}
The following holds:
\begin{enumerate}
\item There exists a smooth flow-invariant principal $H$-subbundle $F\E \supset Q \to \M$ such that the restriction of $(\Phi_t)_{t \in \R}$ to $Q$ is ergodic. 
\item There exists a natural isomorphism
\[
\ker_{L^2(F\E)} X_{F\E} \simeq L^2(H\backslash\mathrm{SO}(r)),
\]
(and $\ker_{L^2(F\E)} X_{F\E} \simeq L^2(H\backslash\mathrm{U}(r))$ in the complex case).
\end{enumerate}
In particular, $(\Phi_t)_{t \in \R}$ is ergodic if and only if $H=\mathrm{SO}(r)$ (or $\mathrm{U}(r)$ in the complex case).
\end{theorem}

When the structure group of the bundle is semi-simple, ergodicity can be automatically upgraded to mixing. We refer to \cite[Corollary 3.10]{Lefeuvre-23} for a proof of this proposition. \\

Since the transitivity group $H$ acts on $\E_{z_\star} \simeq \F^r$, we can define the set of vectors that are $H$-invariant, namely:
\[
\left(\F^r\right)^{H} := \left\{ v \in \F^r ~|~ \forall h \in H, hv = v\right\}.
\]
The following \emph{non-Abelian Liv\v sic theorem} was proved in \cite[Theorem 1.3]{Cekic-Lefeuvre-21-1} and will play a crucial role.

\begin{theorem}[Non-Abelian Liv\v sic theorem]
\label{theorem:livsic}
The evaluation map
\[
\mathrm{ev} \colon C^\infty(\M,\E) \cap \ker \X \to \left(\F^r\right)^{H}, \qquad \mathrm{ev}(f) := f(z_\star) \in \F^r
\]
is an isomorphism.
\end{theorem}

We also state a useful and straightforward corollary. If $\mathfrak{o} \colon \mathrm{Vect} \to \mathrm{Vect}$ is one of the natural operations on the category of finite-dimensional vector spaces (symmetric, exterior power, tensor product, dual), one can consider the induced representation $\rho_{\mathfrak{o}} \colon \mathbf{G} \to \mathrm{End}(\mathfrak{o}(\F^r))$; observe that $\mathfrak{o}(\rho) = \rho_{\mathfrak{o}}$, i.e. $\rho_{\mathfrak{o}}$ agrees with the natural representation obtained by applying $\mathfrak{o}$ to $\rho$. The transitivity group on $\mathfrak{o}(\E)$ is then given by $H = \mathfrak{o}(\rho)(\mathbf{G})$ (we still use the same letter $H$ in order to avoid cumbersome notation). Hence from Theorem \ref{theorem:livsic} we get that the evaluation map is an isomorphism (here $\X$ is extended to act on $\mathfrak{o}(\E)$ naturally):

\begin{corollary}
\label{corollary:remark}
The evaluation map
\[
\mathrm{ev} \colon C^\infty(\M,\mathfrak{o}(\E)) \cap \ker \X \to \left(\mathfrak{o}(\F^r)\right)^{H}
\]
is an isomorphism.
\end{corollary}

Let us now return to the setting of the geodesic flow on $\M = SM$. Any section $f \in C^\infty(SM,\E)$ admits a decomposition into Fourier modes by using (pointwise in $x \in M$) Fourier analysis in the spherical variable of $SM$. More precisely, there exists a natural (non-negative) \emph{vertical Laplacian} operator $\Delta^{\E}_{\V} \colon C^\infty(SM,\E) \to C^\infty(SM,\E)$ defined locally as
\[
\Delta^{\E}_{\V}(f)(x) = \Delta^{\E}_{\V}\left(\sum_{k = 1}^r f_k(x,v) \mathbf{e}_k(x) \right) = \sum_{k = 1}^r \Delta_{\V} f_k (x,v) \mathbf{e}_k(x),
\]
where $(\mathbf{e}_1,...,\mathbf{e}_r)$ is a local trivialization of $\E = \pi^*E$ (note that the sections can be chosen to depend only on the $x \in M$ variable since $\E$ is a pullback bundle), and $\Delta_{\V}$ is the standard vertical Riemannian Laplacian, that is, the spherical Laplacian in each spherical fibers of the fibration $SM \to M$. Then, as in \eqref{equation:l2} we introduce 
\[
\ker(\Delta_{\V} - k(k+n-1)) = \Omega_k, \quad \ker(\Delta_{\V}^{\E}-k(k+n-1)) = \Omega_k \otimes E,
\]
where $\Omega_k \to M$ is the vector bundle of spherical harmonics of degree $k$ and
\begin{equation}\label{eq:L2-expansion}
L^2(SM,\E) = \oplus_{k \geq 0} L^2(M, \Omega_k \otimes E).
\end{equation}
A section $f \in L^2(SM,\E)$ is said to have \emph{finite Fourier content} (or \emph{degree}) if it has only finitely many non-zero terms in the expansion \eqref{eq:L2-expansion}; we will say that $f$ has \emph{zero Fourier degree} if all the terms corresponding to $k \geq 1$ in \eqref{eq:L2-expansion} vanish or equivalently, $f$ is constant in the fibers of $SM$, that is, there is $\widetilde{f} \in L^2(M, E)$ such that $f = \pi^*\widetilde{f}$. The following proposition allows us to make a bridge between ergodicity of these isometric extensions and the (non-)existence of polynomial maps between spheres. It was first obtained in \cite[Theorem 4.1]{Guillarmou-Paternain-Salo-Uhlmann-16}; recall that $(\E, \nabla^{\E}) = (\pi^*E, \pi^*\nabla^{E})$ and $\X = \nabla_X^{\E}$.

\begin{proposition}[Finiteness of the Fourier degree]
\label{theorem:finite}
Let $(M^{n+1},g)$ be a closed negatively curved Riemannian manifold and let $(E,\nabla^E)$ be as above. Assume that
\[
f \in C^\infty(SM, \E) \cap \ker \X.
\]
Then, $f$ has finite Fourier degree.
\end{proposition}

We refer to \cite[Corollary 4.2]{Cekic-Lefeuvre-Moroianu-Semmelmann-22} for a short self-contained proof. Note that the version of Proposition \ref{theorem:finite} in the case where $(M,g)$ is merely Anosov (without any assumption on the curvature) is an open question. An important remark to keep in mind is that if $f \in C^\infty(SM,\pi^*E) \cap \ker \X$ has zero Fourier degree, then there exists a section $f_0 \in C^\infty(M,E)$ such that $f = \pi^*f_0$ and $\nabla^E f_0 = 0$. In what follows, for the sake of simplicity, we will not distinguish the notation between $f$ and $f_0$.

To encompass both cases (real and complex bundles), introduce the notation $q_{\F}(n)$, where we recall that
\begin{equation}
\label{equation:nb}
q_{\R}(n) := q(n)= q_{\mathrm{SO}}(n) -1, \qquad q_{\C}(n) := \tfrac{1}{2} q(n) = q_{\mathrm{U}}(n) - 1.
\end{equation}
By standard algebra, the representation $\rho$ of the monoid $\mathbf{G}$ admits a splitting
\begin{equation}
\label{equation:splitting}
\E_{z_\star} \simeq \F^r = \oplus^\bot_i \oplus_{j = 1}^{n_i} V_i^{(j)},
\end{equation}
where $V_i^{(j)} \subset \mathbb{F}^r$, $n_i \geq 1$, each $V_i^{(j)} \simeq V_i$ is $H$-invariant and irreducible, and the representations $V_i$ and $V_k$ are isomorphic if and only if $i = k$, see e.g. \cite[Chapter XVII]{Lang-02}. In particular, if $\rho$ is irreducible, then $V_1^{(1)} = \F^r$, $n_1=1$, and we will say that the splitting \eqref{equation:splitting} is \emph{trivial}.

\begin{lemma}
\label{lemma:irreducible}
Let $(M^{n+1},g)$ be a closed negatively curved Riemannian manifold. Let $a \in \mathbf{A}^{\F}_{\leq q_{\F}(n)}$ and assume that the corresponding representation of $\mathbf{G}$ admits the splitting \eqref{equation:splitting}. Then
\[
(E,\nabla^E) = \oplus^\bot_i \oplus_{j = 1}^{n_i} \big(F_i^{(j)},\nabla^{F_i^{(j)}}\big),
\]
where $\E = \oplus_i^\bot \oplus_{j = 1}^{n_i} \pi^* F_i^{(j)}$, $V_i^{(j)} = \big(\pi^*F_i^{(j)}\big)_{z_\star}$, and $\big([F_i^{(j)}], [\nabla^{F_i^{(j)}}]\big) = \big([F_i], [\nabla^{F_i}]\big)$ is irreducible. In particular, if $a \in \mathbf{A}^{\F}_{\leq q_{\F}(n)}$ is irreducible, then the transitivity group $H$ acts irreducibly on $\F^r$ and the splitting \eqref{equation:splitting} is trivial.
\end{lemma}

\begin{proof}
By the non-Abelian Liv\v sic theorem (Theorem \ref{theorem:livsic}) and Corollary \ref{corollary:remark} applied with $\mathfrak{o}=S^2$ (symmetric endomorphisms -- they are the orthogonal projections onto the subspaces $V_i^{(j)}$ in \eqref{equation:splitting}), the splitting \eqref{equation:splitting} yields a flow-invariant smooth splitting of $\E$ over $SM$ as:
\[
\E = \oplus^\bot_i \oplus_{j = 1}^{n_i} \mc{F}_i^{(j)}, \qquad \mathbbm{1}_{\E} = \sum_i \sum_{j=1}^{n_i} \pi_{\mc{F}^{(j)}_i},
\]
where $ \pi_{\mc{F}^{(j)}_i}$ is the orthogonal projection onto the vector bundle $\mc{F}_i^{(j)}$. Note that we detect here $H$-invariant decomposition into subbundles by detecting $H$-invariant symmetric endomorphisms; this is a standard idea in representation theory.

Let $\tau$ be the orthogonal projector onto $\mc{F}_i^{(j)} =: \mc{F}$. Note that by assumption, $\mc{F}$ does not split further, that is, there is no non-trivial flow-invariant subbundle of $\mc{F}$. By Proposition \ref{theorem:finite} applied to $\pi^*S^2 E$, $\tau$ has finite Fourier content. By fixing a point $x \in M$ and identifying $S_xM \simeq \Ss^n$, we get a polynomial map $\tau_x \colon \Ss^n \to \mathrm{Gr}_{\F}(k,r)$, where $k = \mathrm{rank}(\mc{F}) \leq q_{\F}(n)$. Hence, by Lemma \ref{lemma:determine}, we deduce that $\tau_x$ is constant, that is, $\tau$ has zero Fourier degree and thus descends to a parallel orthogonal projector $\tau \in C^\infty(M,S^2 E)$ on a parallel subbundle $F \subset E$ such that $\pi^*F = \mc{F}$. Note that $(F,\nabla^E|_{F})$ is irreducible, otherwise any reduction of it would yield a reduction of $\mc{F}$ upstairs, which is excluded by assumption. This proves the claim.
\end{proof}

\begin{lemma}
\label{lemma:finite}
Let $(M^{n+1},g)$ be a closed negatively curved Riemannian manifold with $n \geq 2$. Assume that $([E],[\nabla^{E}]) \in \mathbf{A}^{\F}_{\leq q_{\F}(n)}$ does not have a finite holonomy group. Then, the transitivity group $H$ is not finite.
\end{lemma}

\begin{proof}
If $H$ is finite, then $Q \to SM$ is a finite covering (where the bundle $Q \subset F \mc{E}$ is given by Theorem \ref{prop:brin}). Since $n \geq 2$, $\pi_1(\Ss^n) = 0$ and thus the long exact sequence of the spherical fibration $\Ss^n \hookrightarrow SM \to M$ yields $\pi_1(M) \simeq \pi_1(SM)$. It implies that that there exists a finite cover $p \colon \widetilde{M} \to M$ such that $S\widetilde{M} \simeq Q$ and $S \widetilde{M} \to SM$ is the induced covering map. Lift the bundle $(E,\nabla^E) \to M$ to $(p^*E,p^*\nabla^E) \to \widetilde{M}$. Set $\widetilde{E} := p^*E$ and consider the induced frame flow on the frames of $\widetilde{\E} := \widetilde{\pi}^*\widetilde{E} \to S\widetilde{M}$ (where $\widetilde{\pi} \colon S\widetilde{M} \to \widetilde{M}$ is the projection). By construction, the transitivity group of this frame flow is trivial (see \cite[Lemma 3.12]{Lefeuvre-23} for more details). Hence, by Theorem \ref{theorem:livsic} we can find a global flow-invariant orthonormal basis $\e_1, \dotsc,\e_r \in C^\infty(S\widetilde{M},\widetilde{\E})$. Now, by evaluating the frame $\e := (\e_1, \dotsc,\e_r)$ at a point $x \in M$, we get a map $\e_x \colon \Ss^n \to \mathrm{SO}(r)$ (real case) or $\e_x : \Ss^n \to \mathrm{U}(r)$ (complex case); these maps are polynomial according to Proposition \ref{theorem:finite}. In the latter case, $\det(\e_x) : \Ss^n \to \Ss^1$ is also polynomial so it is constant as $n \geq 2$ (using that $q(2) = 2$ and Lemma \ref{lemma:non-decreasing}); hence, up to normalization, $\e_x : \Ss^n \to \mathrm{SU}(r)$. Since $r \leq q_{\F}(n)$, by Lemma \ref{lemma:determine} this map is constant, and thus the sections $\e_i$ have zero Fourier degree. In turn, this implies that they define parallel sections $\e_i \in C^\infty(\widetilde{M},\widetilde{E})$ over the base $\widetilde{M}$. As a consequence, $([p^*E], [p^*\nabla^E]) = ([\widetilde{M} \times \F^r], [d])$ is the trivial flat bundle equipped with the trivial flat connection. In turn, it implies that $(E,\nabla^E)$ is flat with finite holonomy group. This contradicts the assumption.
\end{proof}

From the discussion above, we can derive our result on the ergodicity of certain isometric extensions in low rank, that is, Theorem \ref{theorem:ergodicity}.

\begin{proof}[Proof of Theorem \ref{theorem:ergodicity}]
We first deal with the complex case. Our aim is to prove that the transitivity group $H \leqslant \mathrm{U}(r)$ satisfies $H = \mathrm{U}(r)$; this would complete the proof according to Theorem \ref{prop:brin}. By construction, $H$ provides an $H$-invariant Lie subalgebra $\mathfrak{h} \leqslant \mathfrak{u}(r)$ of real dimension $k \geq 0$. Note that $\mathfrak{h} \neq 0$ by Lemma \ref{lemma:finite}, that is, $k \geq 1$. By the non-Abelian Liv\v sic theorem (Theorem \ref{theorem:livsic}), we may consider $\tau \in C^\infty(SM,\pi^* S^2 \mathrm{End}_{\mathrm{sk}}(E))$, the flow-invariant orthogonal projection whose value at $z_\star$ is given by the orthogonal projection onto $\mathfrak{h}$ (here $\mathrm{End}_{\mathrm{sk}}$ denote the skew-Hermitian endomorphisms, which are identified with $\mathfrak{u}(r)$ at $z_\star$). By Proposition \ref{theorem:finite}, $\tau$ has finite Fourier content, that is, it is a fiberwise polynomial section. Arguing as in the proof of Lemma \ref{lemma:irreducible}, identifying $\End_{\mathrm{sk}}(E_{x_0}) \simeq \mathfrak{u}(r)$, and using that the real rank of $\mathfrak{u}(r)$ is $r^2$, we get by restricting $\tau$ to an arbitrary sphere $S_{x_0}M \simeq \Ss^n$ (for some $x_0 \in M$) a polynomial map $\Ss^n \to \mathrm{Gr}_{\R}(k,r^2)$. This map needs to be constant by Lemma \ref{lemma:determine} as the condition $r \leq m_{\C}(n)$ implies $r^2 \leq q_{\R}(n)$. Hence, $\tau \in C^\infty(M, S^2 \mathrm{End}_{\mathrm{sk}}(E))$ defines a parallel section on $M$. By assumption, the holonomy group of $(E,\nabla^E)$ is equal to $\mathrm{U}(r)$; it gives rise to the action of the holonomy group on $\End_{\mathrm{sk}}(E_{x_0})$ and this is precisely the adjoint action of $\mathrm{U}(r)$ on $\mathfrak{u}(r)$ after the identification $\End_{\mathrm{sk}}(E_{x_0}) \simeq \mathfrak{u}(r)$. The adjoint representation splits irreducibly as $\mathfrak{u}(r) = \mathfrak{su}(r) \oplus i\R I_r$; we thus obtain that $\tau(z_\star) \neq 0$ is either $\pi_{\mathfrak{su}(r)}, \pi_{i\R I_r}$ or $\mathbbm{1}$. Observe that the last possibility $\tau(z_\star) = \mathbbm{1}$ implies that $H = \mathrm{U}(r)$, that is, $(\Phi_t)_{t \in \R}$ is ergodic on $F\E$ by Theorem \ref{prop:brin}. As a consequence, it suffices to rule out the first two cases.

If $\tau(z_\star) = \pi_{\mathfrak{su}(r)}$, then $H$ is a finite disjoint union of copies of $SU(r)$, and the effective action of $H$ on $\det \C^r = \Lambda^r \C^r$ is that of a finite Abelian group $\Z_p$ (for some $p \geq 1$) and thus $H$ acts trivially on $(\det \C^r)^{\otimes p}$. In turn, by the non-Abelian Liv\v sic theorem (Theorem \ref{theorem:livsic}) and Corollary \ref{corollary:remark} (applied with $\mathfrak{o}(E) := (\det E)^{\otimes p}$), the line bundle $\pi^* (\det E)^{\otimes p}$ is \emph{transparent} over $SM$ (in the terminology of Paternain \cite{Paternain-09}), that is, there exists a flow-invariant section $s \in C^\infty(SM, \pi^* (\det E)^{\otimes p})$ and this has finite Fourier content by Proposition \ref{theorem:finite}. Restricting to a point $x_0 \in M$, we thus a get a polynomial map $\Ss^n \to \Ss^1$ (where $\Ss^1$ is the unit circle of $(\det E_{x_0})^{\otimes p}$) and this map needs to be constant since $n \geq 2$. Hence $s \in C^\infty(M,(\det E)^{\otimes p})$ defines a parallel section on the base $M$, that is, $(\det E)^{\otimes p}$ is the trivial flat line bundle and $(E,\nabla^E)$ thus admits a holonomy reduction, which contradicts the assumptions.

If $\tau(z_\star) = \pi_{i\R I_r}$, we can consider a finite cover $p \colon \widetilde{M} \to M$ as in the proof of Lemma \ref{lemma:finite}, equipped with the pullback bundle $(p^*E,p^*\nabla^E)$, so that the associated transitivity group $\widetilde{H}$ is connected (see also \cite[Lemma 3.12]{Lefeuvre-23}), and thus $\widetilde{H}=\mathrm{U}(1)$. The induced representation of $\widetilde{H}$ on $\C^r$ splits as a sum of irreducible representations $\C^r = \C \oplus \dotsb \oplus \C$ and thus by the non-Abelian Liv\v sic theorem (Theorem \ref{theorem:livsic}), there exists $r$ smooth flow-invariant complex line bundles $\mc{L}_1, \dotsc, \mc{L}_r$ such that $\widetilde{\pi}^*(p^*E) = \mc{L}_1 \oplus \dotsb \oplus \mc{L}_r$ (recall $\widetilde{\pi} \colon S\widetilde{M} \to \widetilde{M}$ is the projection). The orthogonal projection $\pi_{\mc{L}_i} \in C^\infty(S\widetilde{M}, \widetilde{\pi}^*(S^2 p^*E))$ onto each factor is flow-invariant so it is fiberwise polynomial by Proposition \ref{theorem:finite} and thus we obtain a polynomial map $\Ss^n \to \mathrm{Gr}_{\C}(1,r)$. The assumption $r \leq m_{\C}(n)$ implies $r^2 \leq q_{\R}(n)$ and using that $q_{\R}(n)$ is even, this also implies that $r \leq q_{\R}(n)/2$ (arguing separately for the cases $n = 2, 3$ and $n \geq 4$). By Lemma \ref{lemma:determine}, the inequality $r \leq q(n)/2$ then implies that these polynomial maps are constant, that is, $\pi_{\mc{L}_i} \in C^\infty(\widetilde{M},S^2 p^*E)$ is a parallel section on the base $\widetilde{M}$ and $p^*E = \mc{L}_1 \oplus \dotsb \oplus \mc{L}_r$ splits as a sum a parallel complex line bundles. Hence, $(p^*E, p^*\nabla^E)$ admits a holonomy reduction to $\mathrm{U}(1) \times \dotsb \times \mathrm{U}(1) \subset \mathrm{U}(r)$. In turn, this implies that $(E,\nabla^E)$ admits a holonomy reduction, which contradicts the assumption if $r \geq 2$ (if $r = 1$ there is nothing to prove since $\mathfrak{su}(1) = \{0\}$). \\

Let us now deal with the real case. Further assume that $r \neq 1,2,4$; we will deal with the remaining cases afterwards. As in the complex case, the transitivity subgroup $H \leqslant \mathrm{SO}(r)$ provides an $H$-invariant Lie algebra $0 \neq \mathfrak{h} \leqslant \mathfrak{so}(r)$. By the non-Abelian Liv\v sic theorem (Theorem \ref{theorem:livsic}) and Corollary \ref{corollary:remark} applied with $\mathfrak{o}=S^2\Lambda^2$, we get a flow-invariant projector $\tau \in C^\infty(SM, \pi^* S^2 \Lambda^2 E)$ onto a smooth flow-invariant subbundle of $\pi^*\Lambda^2 E$ of rank $k \geq 1$ (whose restriction to $z_\star$ is identified with $\mathfrak{h}$). As before, using that the rank of $\Lambda^2 E$ is $\tfrac{1}{2}r(r-1)$, we get by restriction of $\tau$ to a sphere $S_{x_0}M \simeq \Ss^n$ a polynomial map $\tau: \Ss^n \to \mathrm{Gr}_{\mathbb{R}}(k, \tfrac{1}{2}r(r-1))$. By assumption $r \leq m_{\R}(n)$ implies that $\tfrac{1}{2}r(r-1) \leq q_{\R}(n)=q(n)$ and thus this map is constant by definition of $q(n)$ and Lemma \ref{lemma:determine}. As a consequence, $\tau \neq 0$ is of degree zero and descends on the base to an orthogonal parallel projector $\tau \in C^\infty(M,S^2\Lambda^2 E)$ onto a rank $k$ subbundle. But $(E,\nabla^{E})$ has no holonomy reduction by assumption, so its holonomy group is $\mathrm{SO}(r)$ and since $\mathrm{SO}(r)$ acts irreducibly on $\Lambda^2 \R^r$ (since $\mathrm{SO}(r)$ is simple if $r \neq 1,2,4$ and $\Lambda^2\mathbb{R}^r$ is isomorphic to the adjoint representation), we obtain that $\tau = \mathbbm{1}$ is the identity, that is, $\mathfrak{h} = \mathfrak{so}(r)$ and $H = \mathrm{SO}(r)$. By Theorem \ref{prop:brin}, we conclude that the frame flow $(\Phi_t)_{t \in \R}$ is ergodic on $F\E$.

For $r=4$, we need to slightly adapt the argument. First of all, observe that the assumption $r = 4 \leq m_{\R}(n)$ yields $q(n) \geq 6 \geq q(4) = 4$, that is, $n \geq 4$ as $n \mapsto q(n)$ is non-decreasing by Lemma \ref{lemma:non-decreasing}. As before, up to taking a finite cover, we can also directly assume that $H$ is connected. Now, $\Lambda^2 \R^4$ splits as $\Lambda^2 \R^4 = \Lambda^+ \R^4 \oplus \Lambda^- \R^4$, the space of self-dual and anti self-dual $2$-forms, which give non-isomorphic irreducible representations of $\mathrm{SO}(4)$ (see Proposition \ref{proposition:2n} below where this is further discussed). Hence, the section $\tau \neq 0$ is either $\mathbbm{1}$ (in which case, the frame flow is ergodic) or $\pi_{\Lambda^\pm\R^4}$, one of the two orthogonal projections onto $\Lambda^\pm\R^4$. In both cases, since $H$ is connected, we obtain that $H$ is equal to one of the two $\mathrm{SU}(2)$ factors of $\mathrm{SO}(4) \simeq \mathrm{SU}(2) \times \mathrm{SU}(2)/\Z_2$. Hence, $H$ acts trivially on either $\Lambda^+ \R^4$ or $\Lambda^- \R^4$ and, without loss of generality, we can assume that it acts trivially on $\Lambda^- \R^4$. Applying once again Theorem \ref{theorem:livsic}, we deduce that, in particular, $\pi^* \Lambda^- E$ admits a flow-invariant section $s \in C^\infty(SM,\pi^*\Lambda^-E)$. This yields a polynomial mapping $\Ss^n \to \Ss^2$ which is necessarily constant since $n \geq 4$ (by the preliminary remark) so $s \in C^\infty(M,\Lambda^- E)$ is a parallel section. But then, this contradicts the assumption that $(E,\nabla^E)$ has no holonomy reduction.

For $r=2$, either $H = \mathrm{SO}(2)$ (ergodicity) or $H$ is a finite Abelian group, in which case we can directly assume (up to taking a finite cover of $M$) that $H = \left\{0\right\}$. But then, there exists by Theorem \ref{theorem:livsic} a flow-invariant section $s \in C^\infty(SM,\pi^*E)$ which must be of Fourier degree $0$ by the same arguments as before, that is, $s \in C^\infty(M,E)$ is parallel. In turn, this implies that $(E,\nabla^E) = (\R^2,d)$ is the trivial flat bundle, which contradicts the holonomy assumption. Eventually, the case $r=1$ is empty. This concludes the proof.
\end{proof}

If $\E$ is a real vector bundle over $SM$, one can also consider the lift $(\Phi_t)_{t \in \R}$ of the geodesic flow to $S \E$, the unit sphere bundle of $\E$ over $SM$, given by parallel transport of sections of $\E$ along geodesic flowlines. We say that the holonomy group of $(E, \nabla^{E})$ has \emph{irreducible identity component} if $\mathrm{Hol}_0(E,\nabla^E) \leqslant \mathrm{Hol}(E,\nabla^E)$, the connected component of the holonomy group $\mathrm{Hol}(E,\nabla^E)$ containing the identity (that is, the holonomy group obtained by restricting to homotopically trivial loops), acts irreducibly on $\R^r$ or $\C^r$ in the real or complex cases, respectively.

\begin{corollary}
Let $(M^{n+1},g)$ be a closed negatively curved Riemannian manifold with $n \geq 4$. Let $(E,\nabla^{E}) \in \mathbf{A}^{\R}_{\leq 4}$ be an orthogonal connection on a real vector bundle $E \to M$ such that $r := \mathrm{rank}(E) \leq 4$. Further assume that the holonomy group of $(E, \nabla^{E})$ has irreducible identity component. Then $(\Phi_t)_{t \in \R}$ is ergodic on $S \E$.
\end{corollary}

\begin{proof}
We deal with the case $r=4$, the other cases being similar. Denote by $H_0 \leqslant H$ the identity component of $H$. By \cite[Lemma 3.3]{Cekic-Lefeuvre-Moroianu-Semmelmann-21} there exists a finite Riemannian cover $p: (\widetilde{M}, \widetilde{g}) \to (M, g)$ equipped with the pullback bundle and connection $(\widetilde{E}, \nabla^{\widetilde{E}}) = (p^*E, p^*\nabla^E)$, denoting $(\widetilde{\E}, \nabla^{\widetilde{\E}})$ the pullback to $S\widetilde{M}$ and the induced flow on $S\widetilde{\E}$ by $(\widetilde{\Phi}_t)_{t \in \mathbb{R}}$, such that the transitivity group of $(\widetilde{\Phi}_t)_{t \in \mathbb{R}}$ is equal to $H_0$. Note that ergodicity of this flow implies that of $(\Phi_t)_{t \in \mathbb{R}}$. Observe that there is a double covering homomorphism
\[
	\Pi: \mathrm{SU}(2) \times \mathrm{SU}(2) \to \mathrm{SO}(4), \quad \Pi(a, b) q = a q b^{-1}, \,\,q \in \mathbb{R}^4,
\]
where we think of $\mathrm{SU}(2)$ as the group of unit quaternions acting on $\mathbb{R}^4$; we have $\ker \Pi = \{(1, 1), (-1, -1)\} \simeq \mathbb{Z}_2$. Let $H_0' := \Pi^{-1} H_0$ and consider the left and right projections $H_{0L}' := \pi_L(H_0')$ and $H_{0R}' := \pi_R(H_0')$ onto $\mathrm{SU}(2)$ factors, respectively. 

If $H_{0L}'$ or $H_{0R}'$ project onto $\mathrm{SU}(2)$ then by the classification in \cite[page 170, Cases (1)--(10)]{Mendes-75} we know that either: 1) $H_0$ acts transitively on $\mathbb{S}^3$, in which case $(\widetilde{\Phi}_t)_{t \in \mathbb{R}}$ is ergodic by \cite[Theorem 1]{Lefeuvre-23}, or 2) the representation of $H_0$ on $\mathbb{R}^4$ is reducible. (Note that reducibility of the $H_0$ representation on $\mathbb{R}^4$, as well as transitivity of the $H_0$-action on $\mathbb{S}^3$ by \cite[Proposition, page 167]{Mendes-75}, is encoded in the \emph{P\'olya function} which contains the data of invariant vectors inside symmetric trace-free powers $\mathrm{Sym}_0^k (\mathbb{R}^4)$ for $k \geq 1$. Transitivity is equivalent to the \emph{P\'olya function} being equal to the constant function $1$.) In the case 2), by Lemma \ref{lemma:irreducible}, together with the fact that $q_{\R}(n) \geq q_{\R}(4) = 4$ by Lemma \ref{lemma:non-decreasing} and \eqref{equation:q}, we obtain that $(\widetilde{E}, \nabla^{\widetilde{E}})$ is reducible, hence contradicting the fact that $(E, \nabla^{E})$ has irreducible identity component (as $\mathrm{Hol}_0(E,\nabla^E) \leqslant \mathrm{Hol}(\widetilde{E}, \nabla^{\widetilde{E}})$).

Otherwise, $H_{0L}'$ and $H_{0R}'$ are both at most $1$-dimensional so $\dim H_0 \leq 2$ and therefore $H_0$ is isomorphic to either the trivial group, $\mathbb{S}^1$, or $\mathbb{S}^1 \times \mathbb{S}^1$, and hence has again a reducible action on $\mathbb{R}^4$. As before, this contradicts the irreducibility of the identity component of $\mathrm{Hol}(E, \nabla^{E})$. This completes the proof.
\end{proof}

For surfaces, the situation is slightly different but it involves studying the limiting case where the rank has the dimension of the base. If $(M^2,g)$ is a (closed oriented) Riemannian surface, let $(T_{\mathbb{C}}^*M)^{0, 1}:= \kappa \to M$ be the canonical line bundle. Using the Gysin exact sequence, it can be checked that $\E := \pi^*\kappa \to SM$ is trivial. Moreover, letting $\nabla^{\mathrm{LC}}$ be the Levi-Civita connection on $\kappa$, the induced frame flow $(\Phi_t)_{t \in \R}$ on $F\E \to SM$ is trivial in the sense that it has trivial transitivity group and is thus conjugate to the flow $(\varphi_t, \mathbf{1}_{\mathrm{U}(1)})_{t \in \R}$ on $SM \times \mathrm{U}(1) \simeq F\E \simeq S\E$. In the terminology of Paternain \cite{Paternain-09}, the pair $k := (\kappa, \nabla^{\mathrm{LC}}) \in \mathbf{A}^{\C}_1$ defines a \emph{transparent connection}, that is, its holonomy along every closed geodesic is trivial.

\begin{corollary}
\label{corollary:surfaces}
Let $(M^2,g)$, be a closed Anosov Riemannian surface. Let $a := (E,\nabla^{E}) \in \mathbf{A}^{\C}_{1}$ be a unitary connection on a complex line bundle $E \to M$. Then the frame flow $(\Phi_t)_{t \in \R}$ on $F\E \to SM$ is ergodic, unless there exists $p \in \Z_{\geq 0}, m \in \Z$ such that $a^{\otimes p} = k^{\otimes m}$. 
\end{corollary}

Equivalently, the last equality can be written $a = ((\kappa^{\otimes m})^{\otimes 1/p}, ((\nabla^{\mathrm{LC}})^{\otimes m})^{\otimes 1/p})$, where the index $\otimes 1/p$ stands for the $p$-th unit root.
Topologically, the $p$-th unit root of the line bundle $\kappa^{\otimes m} \to M$, when it exists, is always unique. Actually, it exists if and only if the first Chern class of $\kappa^{\otimes m}$ is divisible by $p$, that is, $m c_1(\kappa) = m (2g-2)$ is divisible by $p$ (where $g$ is the genus of $M$), and it is then given by the unique complex line bundle over $M$ whose first Chern class is $m(2g-2)/p$. However, $((\nabla^{\mathrm{LC}})^{\otimes m})^{\otimes 1/p}$ is not unique and the choice of such an $p$-th root of the connection is parametrized by $H^1(M,\Z/p\Z) = (\Z/p\Z)^{2g}$ (see \cite[Section 21]{Forster-81}).

\begin{proof}
The proof follows from the above description and earlier observations due to Paternain, see \cite[Theorems 3.1 and 3.2]{Paternain-09}.
Indeed, assume that the frame flow $(\Phi_t)_{t \in \R}$ is not ergodic. Then the transitivity group $H \leqslant \mathrm{SO}(2)$ is finite, that is, $H = \Z/p\Z$ for some $p \in \Z_{\geq 1}$. But then the frame flow $(\Phi_t^{\otimes p})_{t \in \R}$ induced by $a^{\otimes p}$ has trivial transitivity group. This is equivalent to saying that $a^{\otimes p}$ is transparent (the holonomy is trivial along every closed geodesic loop). Now, Paternain \cite{Paternain-09} classified all transparent connections on complex line bundles over Anosov surfaces and found that they are precisely given by $\left\{k^{\otimes m} ~|~ m \in \Z\right\}$. 
\end{proof}

\section{Geodesic Wilson loop operator. Isospectral connections.}

\label{section:spectral}

In this section, we prove Theorem \ref{theorem:main}. Let $(M^{n+1},g)$ be a smooth closed connected Riemannian manifold with Anosov geodesic flow and simple length spectrum. 
Consider $a_1, a_2 \in \A$, two connections with same spectrum $\mathbf{S}(a_1) = \mathbf{S}(a_2)$ (note that we do not require that the vector bundles are the same; they could have different ranks \emph{a priori}). In order to state the trace formula of Duistermaat-Guillemin \cite{Duistermaat-Guillemin-75, Guillemin-73} applied to the connection Laplacian, we need to introduce some notation. For $k = 1, 2$, we let:
\begin{itemize}
\item $(\lambda_j^{(k)})_{j \geq 0}$ be the eigenvalues of the connection Laplacian $\Delta_{a_k}$, counted with multiplicities, and sorted in increasing order,
\item $\mc{C}$, the set of free homotopy loops of $M$ (for Riemannian manifolds with Anosov geodesic flow, this set is in bijective correspondence with the set of closed geodesics, that is, there exists exactly one closed geodesic $\gamma_g(c)$ in each free homotopy class $c \in \mc{C}$, see \cite[Theorem 3.8.14]{Klingenberg-95}),
\item $\mc{C}^\sharp$, the set of primitive orbits (i.e. going several times around the same geodesic orbit is excluded),
\item $\sharp : \mc{C} \rightarrow \mc{C}^\sharp$, the operator mapping an orbit to its associated primitive orbit,
\item $\ell(\gamma)$, the length of the closed orbit $\gamma$, and $P_\gamma$, the linearized Poincar\'e return map of $\gamma$ (which may be identified with $d\varphi_{\ell(\gamma)}(z): E^u(z) \oplus E^s(z)$ at the point of interest $z \in (\gamma, \dot{\gamma})$),
\item $\mathrm{Hol}_{\nabla^{E_k}}(c)$, the holonomy on $E_k$ with respect to $\nabla^{E_k}$ around the unique closed geodesic $\gamma_g(c)$ in the class $c \in \mc{C}$.
\end{itemize}
The Duistermaat-Guillemin trace formula then reads (when the length spectrum is simple): 
\begin{equation}
\label{equation:trace-formula}
\lim_{t \to \ell(\gamma_g(c))} \left(t-\ell(\gamma_g(c))\right) \sum_{j \geq 0} e^{-i \sqrt{\lambda^{(k)}_j} t} = \dfrac{\ell(\gamma_g(c^\sharp)) \Tr\left(\mathrm{Hol}_{\nabla^{E_{k}}}(c)\right)}{2\pi |\det(\mathbbm{1}-P_{\gamma_g(c)})|^{1/2}} ,
\end{equation}

As a consequence, when the connection Laplacians are isospectral, the left-hand side of \eqref{equation:trace-formula} is the same for both connections $a_1$ and $a_2$, and one obtains that 
\[
\Tr\left(\mathrm{Hol}_{\nabla^{E_1}}(c)\right) = \Tr\left(\mathrm{Hol}_{\nabla^{E_2}}(c)\right), \qquad \forall c \in \mc{C}.
\]
We will simply restrict to primitive loops in $\mc{C}^\sharp$ and rewrite this last quantity using
\begin{equation}
\label{equation:wilson}
\WW : \A \to \ell^\infty(\mc{C}^\sharp), ~~~ \WW(a)(c) := \Tr(\mathrm{Hol}_{\nabla^{E}}(c)),
\end{equation}
called the \emph{geodesic Wilson loop operator} by analogy with field theory. Note that, since we work with orthogonal/unitary connections, the bound $\|\WW(a)\|_{\ell^\infty(\mc{C}^\sharp)} \leq \rk(E)$ holds. In other words, under the assumption that the length spectrum is simple, the trace formula \eqref{equation:trace-formula} shows that isospectrality of the connection Laplacians $\mathbf{S}(a_1) = \mathbf{S}(a_2)$ implies that the geodesic Wilson loop operators are the same, namely, $\WW(a_1)=\WW(a_2)$. Therefore, the question ``boils down'' to studying the geodesic Wilson loop operator and we formulate the following conjecture:

\begin{conjecture}
If $(M,g)$ is a closed odd-dimensional Riemannian manifold with Anosov geodesic flow, then the Wilson loop operator $\WW \colon \mathbf{A} \to \ell^\infty(\mc{C}^\sharp)$ is injective.
\end{conjecture}

As we shall see below in Proposition \ref{proposition:2n}, when $(M,g)$ is even-dimensional, the operator $\WW$ defined in \eqref{equation:wilson} is never injective. Note that in the specific case of negatively curved surfaces, one can take $a_1 := (M\times \C, d)$, the flat trivial line bundle, and $a_2 := (\kappa, \nabla^{\mathrm{LC}})$, the canonical line bundle $\kappa \to M$ equipped with the metric connection. It can be easily checked that $\WW(a_1)= (1, 1, \dotsc, 1, \dotsc) =\WW(a_2)$ (this is a transparent connection in the terminology of Paternain \cite{Paternain-09}, as seen earlier) but it is clear that $\kappa \not \simeq M \times\C$ since the first Chern class of the canonical line bundle is non-zero.

In the following, denote by $\star$ the Hodge star operator on a Riemannian manifold $(M^{2n}, g)$, and introduce $\Lambda^\pm := \{\alpha \in \Lambda^n T^*M \mid \star \alpha = \pm \alpha\}$ (if $n$ is even) and $\Lambda^\pm := \{\alpha \in \Lambda^n T^*_{\C} M \mid \star \alpha = \pm i \alpha \}$ (if $n$ is odd), the vector bundles of self-dual and anti self-dual $n$-forms equipped with the natural Levi-Civita connections $\nabla^\pm$, respectively. Set $a_+ := (\Lambda^+, \nabla^+)$ and $a_- := (\Lambda^-, \nabla^-)$.

\begin{proposition}
\label{proposition:2n}
	Let $(M^{2n}, g)$ be a negatively curved closed Riemannian manifold. Then, the pairs $a_+$ and $a_-$ are trace-equivalent along closed geodesics, that is,
	\[
	\WW(a_+) = \WW(a_-),
	\]
	but $a_+ \neq a_-$.
	\end{proposition}

\begin{proof}
	We start by proving that $a_+$ and $a_-$ are trace-equivalent along closed geodesics. We argue for $n$ even: the proof for $n$ odd is similar up to minor modifications.

Observe that the equality $\WW(\Lambda^+, \nabla^+) = \WW(\Lambda^-, \nabla^-)$ is implied by the existence of an orthogonal map 
	\[
	p \in C^\infty(SM, \Hom(\pi^*\Lambda^+, \pi^*\Lambda^-))
	\]
	that intertwines the operators $\pi^*\nabla^+_X$ and $\pi^* \nabla^-_X$, i.e.
\begin{equation}
\label{eq:equivalence}
	\forall u \in C^\infty(SM, \pi^*\Lambda^+), \quad \pi^*\nabla^+_Xu = p^{-1}\pi^*\nabla^-_X (p u),
\end{equation}
so it suffices to exhibit this map. (Equivalently, \eqref{eq:equivalence} says that $p$ is invariant with respect to the tensor product connection induced by $\pi^*\nabla^\pm$ on $\Hom(\pi^*\Lambda^+, \pi^*\Lambda^-)$ in the $X$-direction.) Indeed, let $\gamma$ be an arbitrary geodesic segment $\gamma \subset M$, tangent to $v_1, v_2$ at the points $x_1, x_2 \in M$, and let $P^\pm: \Lambda^\pm(x_1) \to \Lambda^\pm(x_2)$ be the parallel transport maps along $\gamma$ with respect to $\nabla^\pm$; then \eqref{eq:equivalence} immediately shows that $P^- = p(x_2, v_2) P^+ p^{-1}(x_1, v_1)$. In particular, taking $\gamma$ to be a closed orbit (so $(x_1, v_1) = (x_2, v_2)$) shows that $P^+$ and $P^-$ have equal traces and so $\WW(\Lambda^+, \nabla^+) = \WW(\Lambda^-, \nabla^-)$.

Introduce the commuting orthogonal projections $\Pi^\pm := \frac{\id \pm \star}{2}$ onto $\Lambda^\pm$, i.e. $(\Pi^\pm)^2 = \Pi^\pm$, $\Pi^+ \Pi^- = \Pi^- \Pi^+ = 0$, and $\Lambda^\pm = \Pi^\pm \Lambda^n$. We claim that the map $p$ defined as follows:
	\begin{equation}
	\label{equation:p}
	p(x, v) \alpha := 2\Pi^-(x)\big(v \wedge \star(\alpha \wedge v)\big), \quad \alpha \in \Lambda^+_x,
	\end{equation}
	is flow-invariant and that for every $(x,v) \in SM$, $p(x,v) \colon \Lambda^+_x \to \Lambda^-_x$ is an isometry. (Note that in \eqref{equation:p}, we have implicitly identified $v$ with a $1$-form using the metric.) 

In order to check the flow-invariance \eqref{eq:equivalence}, it suffices to take a geodesic segment $\gamma \subset M$ tangent to $v_1, v_2$ at the points $x_1, x_2 \in M$, respectively, and observe that the parallel transport $P_\gamma$ along $\gamma$ satisfies for an arbitrary $\alpha \in \Lambda^+_{x_1}$:
	\begin{align*}
		p(x_2, v_2) P_\gamma \alpha &= 2\Pi^-(x_2) \big(v_2 \wedge \star(P_\gamma \alpha \wedge v_2)\big) \\
		& = 2P_\gamma \Pi^-(x_1) \big(v_1 \wedge \star(\alpha \wedge v_1)\big) =  P_\gamma p(x_1, v_1) \alpha.
	\end{align*}
	Here in the second equality we used that $P_\gamma$ commutes with $\star$, $\Pi^-$, that it distributes over the wedge product, and that it satisfies $P_\gamma v_1 = v_2$. This completes the proof of the claim.

	The isometry property follows from the following observation: take an orthonormal basis $(\e_i)_{i = 1}^{2n}$ of $T_xM$ such that $v = \e_1$. For an increasing $(n - 1)$-tuple $I = (i_1, i_2, \dotsc, i_{n - 1}) \subset \left\{2, \dotsc, 2n\right\}$, define $\alpha_{I}^\pm := \sqrt{2} \Pi^\pm \e_1 \wedge \e_I$; it is straightforward to check that $(\alpha_I^\pm)_I$ is an orthonormal basis of $\Lambda^\pm_x$. Moreover, a simple computation shows that 
	\[p(x, \e_1) \alpha_I^+ = \sqrt{2} \Pi^-(x) \big(\e_1 \wedge \star \big(\star(\e_1 \wedge \e_I) \wedge \e_1\big)\big) = \sqrt{2} \Pi^-(x) \big(\e_1 \wedge \e_I\big) = \alpha_I^-,\]
	where in the second equality we used $\star \big(\star(\e_1 \wedge \e_I) \wedge \e_1\big) = \e_I$, so $p$ is an isometry. This completes the proof that $\mathbf{W}(a_+) = \mathbf{W}(a_-)$. \\
	
	We now show that $a_+ \neq a_-$. Actually, when $\dim M = 4$, $[\Lambda^+] \neq [\Lambda^-]$, that is, the bundles are not isomorphic: this follows from the fact that their first Pontryagin classes $p^\pm \in H^4(M, \mathbb{Z})$ satisfy the formula $p^+ - p^- = 4\e$, where $0 \neq \e \in H^4(M, \mathbb{Z})$ is the Euler class \cite[Chapter 6, Proposition 5.4(2)]{Walschap-04}, which is strictly positive by the generalised Gauss-Bonnet theorem, see \cite[Theorem 5]{Chern-55}. 
	
However, we can always show that $a_+ \neq a_-$, that is, the connections are not gauge-equivalent. Indeed, the equality $a_+ = a_-$ would entail the existence of an isometry $p \in C^\infty(M,\mathrm{Hom}(\Lambda^+,\Lambda^-))$ such that $p^*\nabla^-(\bullet) = \nabla^+(\bullet) = p^{-1}\nabla^- (p\bullet)$. In particular, restricting to an arbitrary point $x_0 \in M$, we would obtain that the holonomy representations $\rho_\pm$ of the loop space $\Omega M$ (based at $x_0$) in $\mathrm{SO}(\Lambda^\pm_{x_0})$ are conjugate, that is $\rho_+(\gamma)= p^{-1}(x_0) \rho_-(\gamma) p(x_0)$ for every loop $\gamma \in \Omega M$ based at $x_0$. In particular, this holds for all $\gamma \in (\Omega M)_0$, that is, $\gamma \in \Omega M$ such that $[\gamma]=0 \in \pi_1(M,x_0)$. In order to conclude, we then argue distinctly for each possible holonomy groups arising in negative curvature (see Berger's classification \cite{Berger-53} -- Berger's classification applies to simply connected manifolds, so we need to pass to the universal cover of $(M,g)$, hence the restriction to the homotopically trivial loops $(\Omega M)_0$): 

\begin{enumerate}
\item \textbf{Generic case}, $\mathrm{Hol}_0(M)=\mathrm{SO}(2n)$: in this case, the morphisms $\rho_\pm$ factor surjectively through $\mathrm{SO}(T_{x_0}M) \simeq \mathrm{SO}(2n)$, which is impossible since $\Lambda^\pm$ are irreducible non-isomorphic representations of $\mathrm{SO}(2n)$, see \cite[Chapter 6, Section 5.5]{Brocker-Dieck-85}:
\[
\xymatrix{
     & \mathrm{SO}(\Lambda^+_{x_0})  \\
    (\Omega M)_0 \ar[ur]^{\rho_+} \ar@{->>}[r] \ar[dr]_{\rho_-} & \mathrm{SO}(2n) \ar[u] \ar[d] \\
    & \mathrm{SO}(\Lambda^-_{x_0}).
  }
\]
\item \textbf{Kähler case}, $\mathrm{Hol}_0(M)=\mathrm{U}(n) < \mathrm{SO}(2n)$: as in (1), it suffices to show that the restricted representations of $\mathrm{U}(n)$ on $\Lambda^\pm$ are not isomorphic and we will actually show that the restricted representations to the center $\mathrm{U}(1)$ of $\mathrm{U}(n)$ on $\Lambda^\pm$ are not isomorphic. Let $\rho \colon \mathrm{U}(1) \to \mathrm{SO}(2)$ be the standard representation; the standard representation $\mathrm{U}(1) \to \mathrm{End}(\R^{2n})$ is given by $\rho \oplus \dotsb \oplus \rho$ ($n$ times). Recall that $R^+(G,\R)$, the real representation semiring of a group $G$, is equipped with a (semi)ring homomorphism $\lambda_t \colon R^+(G,\R) \mapsto R(G,\R)[[t]]$ given by (see \cite[Chapter 7]{Brocker-Dieck-85})
\[
\lambda_t(\rho) := 1 + \rho t + (\Lambda^2 \rho) t^2 + \dotsb
\]
such that $\lambda_t(\rho_1\oplus\rho_2)=\lambda_t(\rho_1) \cdot \lambda_t(\rho_2)$. Hence, in our case, for $\rho : \mathrm{U}(1) \to \mathrm{SO}(2)$ as above, we get
\begin{equation}
\label{equation:lambdat}
\begin{split}
\hspace{1cm} \lambda_t(\rho \oplus \dotsb \oplus \rho) & = \lambda_t(\rho)^n \\
& = (1+ \rho t + t^2)^n = 1 + \dotsb + c_n(\rho) t^n + \dotsb + t^{2n},
\end{split}
\end{equation}
where $c_n(\rho)$ is a polynomial in the representation $\rho$ such that
\begin{equation}
\label{equation:cn}
c_n(\rho) = \rho^{\otimes n} + \text{lower order terms}.
\end{equation}
The representation of $\mathrm{U}(1)$ on $\Lambda^n \R^{2n}$ is precisely given by the coefficient $c_n(\rho)$ in the expansion \eqref{equation:lambdat} and \eqref{equation:cn} shows that this representation admits a weight $n$ of multiplicity $1$. As a consequence, when $n$ is even, the two real representations $\Lambda^\pm$ cannot be isomorphic.

 When $n$ is odd, this argument needs to be slightly adapted since we need to complexify $\Lambda^n \R^{2n}$ in order to obtain the decomposition $\Lambda^n \R^{2n} \otimes_{\R} \C = \Lambda^+ \oplus \Lambda^-$. Hence, complexifying \eqref{equation:cn}, we see that $\Lambda^n \R^{2n} \otimes_{\R} \C$ is given by $\rho^{\otimes n} \otimes_{\R} \C + \text{l.o.t.} = \rho_{\C}^{\otimes n} + \text{l.o.t.}$, where $\rho_{\C}$ stands for the complexification of $\rho$. Note that $\rho : \mathrm{U}(1) \to \mathrm{SO}(2)$ can already be seen as a $1$-dimensional complex representation which we denote by $\eta$ (in order to avoid confusion) and thus $\rho_{\C} = \eta \oplus \overline{\eta}$ (as complex representations), which yields:
 \[
 \rho^{\otimes n} \otimes_{\R} \C = (\eta \oplus \overline{\eta})^{\otimes n} = \eta^{\otimes n} + \text{terms involving } \overline{\eta}.
 \]
 (Here $\eta^{\otimes n}$ is the $1$-dimensional complex representation associated to the character $z \mapsto z^n$.) Since $\eta^{\otimes n}$ appears with multiplicity $1$, it implies that $\Lambda^+ \neq \Lambda^-$ in this case too.

\item \textbf{Quaternion Kähler case}, $\mathrm{Hol}_0(M) = \mathrm{Sp}(n).\mathrm{Sp}(1) < \mathrm{SO}(4n)$: same argument as in (2) (case $n$ even) by restricting to $\mathrm{U}(1)$.

\item \textbf{Octonionic (or Cayley) hyperbolic plane}, $\mathrm{Hol}_0(M) = \mathrm{Spin}(9) < \mathrm{SO}(16)$: the faithful representation $\mathrm{Spin}(9) \to \mathrm{SO}(16)$ is given by the spinor representation $\Delta_9$ but by \cite[Theorem 1]{Friedrich-01}, it is known that the induced representations $\Lambda^p \Delta_9$ are multiplicity-free (this can also be checked using the LiE program), so $\rho_+ \neq \rho_-$.
\end{enumerate}
This completes the proof.
	\end{proof}

\begin{remark}\rm
	Set $\Lambda^{\mathrm{odd}/\mathrm{even}} := \Lambda^{\mathrm{odd}/\mathrm{even}} T^*M$, the bundle of odd/even differential forms, and consider $a_{\mathrm{odd}/\mathrm{even}} := (\Lambda^{\mathrm{odd}/\mathrm{even}}, \nabla^{\mathrm{LC}}) \in \mathbf{A}^{\mathbb{R}}$, where $\nabla^{\mathrm{LC}}$ is the Levi-Civita connection. Similarly to Proposition \ref{proposition:2n} it is straightforward to see $\WW(a_{\mathrm{odd}}) = \WW(a_{\mathrm{even}})$. If $n = 2$, then $\Lambda^{\mathrm{odd}} \not \simeq \Lambda^{\mathrm{even}}$ since $\Lambda^{\mathrm{odd}} = \Lambda^1 T^*M$ and $\Lambda^{\mathrm{even}} \simeq M\times \mathbb{R}^2$ have distinct Euler classes; if $n > 2$, then $\Lambda^{\mathrm{odd}/\mathrm{even}}$ are in the `stable regime' (they have rank $2^{n - 1} > n$) so they are isomorphic if and only if stably isomorphic. Calculations indicate that $\Lambda^{\mathrm{odd}/\mathrm{even}}$ are equal in K-theory and are hence isomorphic for $n > 2$. However, by a similar argument to Proposition \ref{proposition:2n}, we expect that $a_{\mathrm{odd}} \neq a_{\mathrm{even}}$. Also, guided by the case of $\Lambda^{\mathrm{odd}/\mathrm{even}}$, we may expect that $\Lambda^+ \simeq \Lambda^-$ in higher dimensions.
\end{remark}

The $4$-dimensional case is enlightening and we detail here some computations. The two trace-equivalent pairs $(\Lambda^+, \nabla^+), (\Lambda^-, \nabla^-)$ provide a non-trivial example of a dynamically invariant fiberwise polynomial map over $\Ss^3$ which turns out to be the usual Hopf fibration $\Ss^3 \to \Ss^2$.

\begin{example}
Assume that $\dim M = 4$. When restricted to the sphere $S_xM$ over some $x \in M$ and identifying $\Lambda^\pm_x \simeq \mathbb{R}^3$, the map $p$ defined by \eqref{equation:p} provides a non-constant polynomial mapping $p: \mathbb{S}^3 \to \mathrm{SO}(3)$ which, when restricted to a column, then gives a quadratic polynomial map $\mathbb{S}^3 \to \mathbb{S}^2$. We claim that, in local coordinates, it is given by the standard Hopf fibration $\Ss^3 \to \Ss^2$ (as described in \S\ref{section:polynomial} for instance). Indeed, taking $(\e_1,\e_2,\e_3,\e_4)$, an orthonormal basis of $T_xM$, we can use the following orthonormal basis of $\Lambda^\pm_x$ (we freely identify vectors and $1$-forms via the metric):
	\begin{equation}
	\label{equation:basis}
		\alpha_1^\pm :=  \frac{\e_1 \wedge \e_2 \pm \e_3 \wedge \e_4}{\sqrt{2}}, \alpha_2^\pm :=  \frac{\e_1 \wedge \e_3 \mp \e_2 \wedge \e_4}{\sqrt{2}},  \alpha_3^\pm := \frac{\e_1 \wedge \e_4 \pm \e_2 \wedge \e_3}{\sqrt{2}}.
	\end{equation}
	A quick computation gives an expression for $p(v)\left(\alpha_1^+\right)$ in the basis $(\alpha_1^-,\alpha_2^-,\alpha_3^-)$:
	\[
	\begin{split}
	&p(v)\left( \tfrac{\e_1 \wedge \e_2 + \e_3 \wedge \e_4}{\sqrt{2}} \right) \\
	& = 2\Pi^-\left( v \wedge \star \left( \tfrac{\e_1 \wedge \e_2 + \e_3 \wedge \e_4}{\sqrt{2}}  \wedge v \right) \right) \\
	& = \sqrt{2}\Pi^-\left( (v_1^2+v_2^2) \e_1 \wedge \e_2+(v_2v_3-v_1v_4) \e_1 \wedge \e_3 + (v_1v_3+v_2v_4)\e_1 \wedge \e_4 \right. \\
	& \left. -(v_1v_3+v_2v_4)\e_2 \wedge \e_3 + (v_2v_3-v_1v_4)\e_2 \wedge \e_4 + (v_3^2+v_4^2)\e_3 \wedge \e_4 \right) \\
	& = \left(v_1^2+v_2^2-(v_3^2+v_4^2) \right) \alpha_1^- + 2\left(v_2v_3-v_1v_4\right) \alpha_2^- + 2\left(v_1v_3+v_2v_4 \right) \alpha_3^-.
	\end{split}
	\]
	Writing $z_0:=v_1+iv_2 \in \C,z_1=v_3+iv_4 \in \C$, we thus get 
	\[
	p(v)\left( \tfrac{\e_1 \wedge \e_2 + \e_3 \wedge \e_4}{\sqrt{2}} \right) = (|z_0|^2-|z_1|^2, 2z_0z_1^*),
	\]
	which is the usual expression for the Hopf fibration $\Ss^3 \to \Ss^2$.

\end{example}

We now study the injectivity of the geodesic Wilson loop operator $\WW$ on low-rank vector bundles and prove the following:

\begin{theorem}
\label{theorem:wilson}
Let $n \in \mathbb{Z}_{\geq 2}$ and let $(M^{n+1}, g)$ be a negatively curved closed Riemannian manifold. Then
\[
\WW \colon \mathbf{A}^{\F}_{\leq q_{\F}(n)} \to \ell^\infty(\mc{C}^\sharp)
\]
is injective.
\end{theorem}

Recall here that $q_{\F}(n)$ was introduced in \eqref{equation:nb}. Note that we do not make any assumption on the parity of $n$ in Theorem \ref{theorem:wilson}. As a corollary, we immediately obtain the same result for the spectrum map, namely, Theorem \ref{theorem:main}. Finally, observe that Proposition \ref{proposition:2n} does not contradict Theorem \ref{theorem:wilson} as the low-rank condition is not satisfied there.

\begin{proof}
 Similarly to \S\ref{section:ergodicity}, given $a = ([E],[\nabla^E]) \in \A^{\F}$, one can consider the lift $([\E],[\nabla^{\E}]) := \pi^*([E],[\nabla^E])$ to $SM$ and the induced representation of Parry's free monoid $\rho : \mathbf{G} \to \mathrm{End}(\F^r)$ given by parallel transport of sections along homoclinic orbits with respect to $\nabla^{\E}$. Recall that the character of this representation is defined as $\chi_\rho := \Tr(\rho(\bullet))$. The key point is that equality of the Wilson loop operator $\WW(a_1) = \WW(a_2)$, where $a_i = ([E_i],[\nabla^{E_i}])$ and $r_i := \mathrm{rank}(E_i)$, implies that the representations \eqref{equation:parry-representation} of Parry's free monoid $\rho_i : \mathbf{G} \to \mathrm{SO}(r_i) \simeq \mathrm{SO}\big((\E_i)_{z_\star}\big)$ (or $\mathrm{U}(r_i) \simeq \mathrm{U}\big((\E_i)_{z_\star}\big)$ in the complex case) have the same character for $i = 1, 2$, see \cite[Proposition 3.18]{Cekic-Lefeuvre-21-1}. This is due to the fact that periodic orbits are dense and thus by the shadowing lemma for hyperbolic flows \cite[Chapter 5]{Fisher-Hasselblatt-19}, one can approximate homoclinic orbits by periodic ones. Hence, by standard algebra, these representations are isomorphic, see \cite[Chapter XVII, Corollary 3.8]{Lang-02}. 

In other words, there is an isometry $p_\star \colon {\E_1}_{z_\star} \to {\E_2}_{z_\star}$ such that
\begin{equation}
\label{equation:conj}
\rho_1 = p_\star^{-1} \rho_2 p_\star.
\end{equation}
This implies that the two bundles $\E_1$ and $\E_2$ have the same rank $r = r_1 = r_2$ and that both transitivity groups $H_1$ and $H_2$ are the same (up to conjugacy within $\mathrm{SO}(r)$ or $\mathrm{U}(r)$).

Apply the non-Abelian Liv\v sic theorem (Theorem \ref{theorem:livsic}) with the bundle $\mathrm{Hom}(\E_1, \E_2) = \E_2 \otimes \E_1^*$ equipped with $\nabla^{\mathrm{Hom}(\E_1, \E_2)} := \nabla^{\E_2} \otimes \nabla^{\E_1^*}$: indeed, using Corollary \ref{corollary:remark}, the induced representation
\[
\rho_{\mathrm{Hom}(\E_1, \E_2)} \colon \mathbf{G} \to \End\big(\mathrm{Hom}({\E_1}_{z_\star},{\E_2}_{z_\star}) \big)
\]
given by
\[
\rho_{\mathrm{Hom}(\E_1, \E_2)}(g) p := \rho_2(g)^{-1} p \rho_1(g),
\]
admits $p_\star$ as a fixed point by \eqref{equation:conj}. Hence, by Theorem \ref{theorem:livsic}, one obtains a flow-invariant isometry $p \in C^\infty(SM, \mathrm{Hom}(\E_1, \E_2))$ such that $p(z_\star)=p_\star$, that is, $\X p=0$ where $\X = \nabla^{\mathrm{Hom}(\E_1, \E_2)}_X$. The fact that $p$ is a fiberwise isometry over $SM$ can be seen as follows: using $\X p = 0$, one gets 
\[
	\nabla_X^{\End(\E_1)} (p^*p) = (\X p)^* p + p^* \X p = 0,
\]
where $\nabla^{\End(\E_1)}$ is the endomorphism connection on $\End(\E_1)$. Thus by invoking Corollary \ref{corollary:remark}, and using that $p^*p(z_\star) = \mathbbm{1}_{z_\star}$, we get that $p^* p \equiv \mathbbm{1}$, which shows that $p$ is a fibrewise isometry.

In particular, this implies that $\E_1 = \pi^*E_1 \simeq \pi^*E_2 = \E_2$ are isomorphic over $SM$. If we further assume that $(M,g)$ has negative sectional curvature, we can apply Proposition \ref{theorem:finite}, to deduce that $p$ has finite Fourier degree. Restricting to some $x \in M$, the bundles $\E_1$ and $\E_2$ over $S_xM \simeq \Ss^n$ are trivial and can thus be identified with $\F^r$. Hence $p$ defines a polynomial map $p \colon \Ss^n \to \mathrm{SO}(r)$ (or $\mathrm{SU}(r)$ in the complex case -- as in the proof of Lemma \ref{lemma:finite}, one can reduce from $\mathrm{U}(r)$ to $\mathrm{SU}(r)$) which by the assumption $r \leq q_{\F}(n)$ and Lemma \ref{lemma:determine} has degree zero, so $p$ descends to a smooth isometry in $C^\infty(M, \mathrm{Hom}(E_1, E_2))$, parallel with respect to the induced connection on the homomorphism bundle, that is, $\nabla^{\mathrm{Hom}(E_1, E_2)} p = 0$. In turn, this gives the gauge-equivalence of $a_1$ and $a_2$, completing the proof.
\end{proof}

\bibliographystyle{alpha}
\bibliography{Biblio}

\end{document}